\documentclass[reqno]{amsart}
\usepackage{amsmath, amssymb, amsthm, epsfig}
\usepackage[hidelinks]{hyperref}
\usepackage{latexsym}
\usepackage{url}
\usepackage[mathscr]{euscript}
\usepackage{array}
\usepackage{wrapfig}
\usepackage{multirow}
\usepackage{tabularx}

\usepackage{color}
\usepackage{fullpage}
\usepackage{setspace}

\def\today{\ifcase\month\or
	January\or February\or March\or April\or May\or June\or
	July\or August\or September\or October\or November\or December\fi
	\space\number\day, \number\year}

\DeclareMathOperator{\supp}{\mathrm{supp}}
\newtheorem{problem}{Extremal Problem}
\newtheorem{theorem}{Theorem}
\newtheorem*{theoremA}{Theorem A}
\newtheorem{conjecture}[theorem]{Conjecture}
\newtheorem{lemma}[theorem]{Lemma} 
\newtheorem{proposition}[theorem]{Proposition}

\newtheorem{corollary}[theorem]{Corollary}
\theoremstyle{definition}

\theoremstyle{remark}
\newtheorem*{remark}{Remark}

\newcommand{\mc}{\mathcal}

\newcommand{\nn}{\mathfrak{n}}

\newcommand{\C}{\mathbb{C}}
\newcommand{\R}{\mathbb{R}}

\newcommand{\Q}{\mathbb{Q}}
\newcommand{\Z}{\mathbb{Z}}

\newcommand{\hh}{\tfrac12}

\newcommand{\dt}{\text{\rm d}t}
\renewcommand{\d}{\text{\rm d}}

\newcommand{\ov}{\overline}

\newcommand{\im}{{\rm Im}\,}
\newcommand{\re}{{\rm Re}\,}
\newcommand{\qq}{\mathfrak{q}}
\newcommand{\pp}{\mathfrak{p}}

\newcommand{\cond}{\mathfrak{f}_\chi}

\newcommand{\norm}[1]{\left\lVert#1\right\rVert}

\title{Fourier optimization and quadratic forms}

\author[Chirre]{Andr\'{e}s Chirre}
\author[Quesada-Herrera]{Oscar E. Quesada-Herrera}
\address{Department of Mathematical Sciences, Norwegian University of Science and Technology, NO-7491 Trondheim, Norway}
\email{carlos.a.c.chavez@ntnu.no }
\address{IMPA - Instituto Nacional de Matem\'{a}tica Pura e Aplicada - Estrada Dona Castorina, 110, Rio de Janeiro, RJ, Brazil 22460-320}
\email{oscarqh@impa.br}

\begin{document}

	\allowdisplaybreaks
	\numberwithin{equation}{section}
	\maketitle
	
	\begin{abstract} We prove several results about integers represented by positive definite quadratic forms, using a Fourier analysis approach. In particular, for an integer $\ell\ge 1$, we improve the error term in the partial sums of the number of representations of integers that are a multiple of $\ell$. This allows us to obtain unconditional Brun-Titchmarsh-type results in short intervals, and a conditional Cramér-type result on the maximum gap between primes represented by a given positive definite quadratic form.

	\end{abstract}

	\section{Introduction}
	In this paper, we combine tools from Fourier analysis, analytic number theory, and algebraic number theory to prove a number of new estimates related to integers represented by positive definite quadratic forms. In particular, we improve some results given by Zaman \cite[Proposition 7.1 and Theorem 1.4]{Zam}, concerning these types of estimates. As an application, assuming the Generalized Riemann Hypothesis, we establish a Cram\'er-type result, extending the method developed by Carneiro, Milinovich, and Soundararajan \cite{CMS}.
	
	\subsection{Background} A classical problem in number theory is to understand the distribution of primes represented by positive definite quadratic forms. The survey \cite{Cox} by D. A. Cox is the classical reference on the subject, describing some of the historical milestones of its study and showing how it leads to class field theory.
	
	An integral quadratic form in two variables is a function defined by
	$$f(u,v)=au^2+buv+cv^2,$$
	where $a,b,c \in \Z$. Its discriminant $-D$ is given by $-D=b^2-4ac$. For simplicity, we refer to a form (or quadratic form) as a function $f$ defined in this way. We say that $f$ is positive definite if $D>0$, and $f$ is primitive if its coefficients $a, b$, and $c$ are relatively prime. In the set of primitive forms, we define an equivalence relation in the following way: $f$ and $g$ are properly equivalent if there are integers $p,q,r,s$ such that
	$$
	f(u,v)= g(pu+qv,ru+sv), \hspace{0.2cm} \mbox{and} \hspace{0.2cm} ps-qr=1.
	$$
	Note that two properly equivalent forms have the same discriminant. A primitive positive definite form $f$ is reduced if $|b|\leq a\leq c$ and if, in addition, when $|b|=a$ or $a=c$, then $b\geq 0$. Classical theorems in the theory of quadratic forms (see \cite[Theorem 2.8 and Theorem 2.13]{Cox}) establish that every primitive positive definite form is properly equivalent to a unique reduced form. Moreover, for each $D>0$, the number of classes of primitive positive definite forms of discriminant $-D$ is finite, and it is equal to the number of reduced forms of discriminant $-D$. This number is called the class number and it is denoted by $h(-D)$.

	\subsection{Congruence sums}
	An integer $n$ is represented by the quadratic form $f$ if there is $(u,v)\in \Z^2$ such that $n=f(u,v)$. For $n\geq 0$ an integer, define
	\begin{equation*}
	r_f(n)=\#\{(u,v)\in \Z^2:f(u,v)=n\},
	\end{equation*}
	that is, the number of representations of $n$ by $f$. Motivated by applications using sieve theory, we are interested in estimating the congruence sums
	\begin{equation}\label{eq:congruenceSum}
	\sum_{\substack{n\le x\\ \ell|n}} r_f(n),
	\end{equation}
	where $x\geq1$ is a real number and $\ell\geq 1$ is an integer. In the case $\ell=1$ and $f(u,v)=u^2+v^2$, the congruence sum \eqref{eq:congruenceSum} corresponds to the classical Gauss circle problem.\footnote{\,\,\, For a survey of this problem, see \cite[Section 2.7]{grosswald} and \cite{BER}.} Here, Gauss used a lattice point counting argument to prove that \eqref{eq:congruenceSum} has the asymptotic formula $\pi x+O(x^{1/2})$. Later, Sierpi\'nski improved the error term to $O(x^{1/3})$ using ideas from Voronoi's work on the Dirichlet divisor problem. Afterward, Landau \cite[Treatise I]{landau} extended this asymptotic formula to positive definite quadratic forms (still in the case $\ell=1$), with error term $O(x^{1/3})$, but without making explicit the dependence on $f$ in this error term. For the case where $\ell\geq 1$ is a squarefree\footnote{\,\,\, As we shall see in \eqref{15_342}, our result also holds for an arbitrary integer $\ell\ge 1$, with an adequate function $g(\ell)$ and perhaps a modification in the error term.} integer, we prove the following result.

	\begin{theorem} \label{16_43_02_08}
		Let $f(u,v)=au^2+buv+cv^2$ be a reduced positive definite quadratic form of discriminant $-D$ and let $\ell\geq 1$ be a squarefree integer. Then, for $x\geq D^2$ we have
		\begin{align}  \label{15_56}
		\sum_{\substack{1\leq n\leq x\\ \ell|n}}r_f(n)= \frac{2\pi}{\sqrt{D}} \,g(\ell)\,x+ O\Bigg(\dfrac{\tau(\ell)\,\ell}{D^{1/6}}x^{1/3}+\frac{\tau(\ell)\,\ell^{5/2}D^{3/4}}{a^{7/4}}x^{1/4}\Bigg),
		\end{align}
		where $g$ is the multiplicative function defined by
		\begin{align*} 
		g(p)=\dfrac{1}{p}\bigg(1+\chi(p)-\dfrac{\chi(p)}{p}\bigg)
		\end{align*}
		for all primes $p$, $\chi=\chi_{-D}=\left(\frac{-D}{\cdot}\right)$ is the corresponding Kronecker symbol, and $\tau$ is the divisor function.
	\end{theorem}

\noindent Theorem \ref{16_43_02_08} improves a result of Zaman \cite[Proposition 7.1]{Zam}, whose error term is of magnitude $x^{1/2}$. Note that, when $\ell=1$, we recover Landau's result, with an explicit dependence on $f$ in the error term. 

\smallskip
As we shall see in the next section, a direct application of Selberg's sieve allows us to use Theorem \ref{16_43_02_08} to obtain upper bounds for the number of primes represented by $f$, in short intervals.
\subsection{Brun-Titchmarsh-type result}  Assume that $f$ is a primitive positive definite quadratic form. For $x\geq 1$, let $\pi_f(x)$ be the number of primes represented by $f$ up to $x$, i.e.,
	$$
	\pi_f(x)=\#\{p\leq x:p=f(u,v) \hspace{0.2cm} \mbox{for some}\hspace{0.15cm}  (u,v)\in \Z^2\}.
	$$
	The classical result for $\pi_f(x)$ goes back to de la Vall\'ee Poussin (see, for instance \cite{Na}), and establishes that, as $x\to\infty$, 
	\begin{align*}
	\pi_f(x) \sim \dfrac{\delta_f\,x}{h(-D)\log x},
	\end{align*}
	where
	\begin{equation}\label{20_39}
	\delta_f= \left\{
	\begin{array}{ll}
	\dfrac{1}{2} ,     & \mathrm{if\ } f(u,v) \, \, \mbox{is properly equivalent to\ } f(u,-v);  \\
	1     & \mathrm{otherwise.\ } 
	\end{array}
	\right.
	\end{equation}
	Assuming the Generalized Riemann Hypothesis (GRH), we also have (see \cite{LO})
	\begin{align} \label{14_17} 
	\pi_f(x)=\dfrac{\delta_f\,\mbox{Li}(x)}{h(-D)}+ O\big(x^{1/2}\log(Dx)\big),
	\end{align} 
	for $x\geq 2$, where 
	$$\mbox{Li}(x)=\int_{2}^{x}\frac{1}{\log t}\,\dt.$$ Recently, Thorner and Zaman \cite[Corollary 1.3]{TZam} established a Brun-Titchmarsh result, improving upon the Chebotarev version given by Lagarias-Montgomery-Odlyzko \cite{LMO}. Unconditionally, they showed that, for $D$ sufficiently large, 
	\begin{align} \label{17_17}
	\pi_f(x)<\dfrac{2\,\delta_f\,\mbox{Li}(x)}{h(-D)}, \hspace{0.2cm} \mbox{for}\hspace{0.1cm}  x\geq D^{700}.
	\end{align}
	We want to establish a result similar to \eqref{17_17} for primes in short intervals.\footnote{\,\,\, Montgomery and Vaughan \cite[Theorem 2]{MV} gave a classical version for primes in arithmetic progressions, in short intervals.} For instance, 
	if we assume GRH, from \eqref{14_17} we get that 
	\begin{align*} 
	\pi_f(x)-\pi_f(x-y) \ll \dfrac{\delta_f\,y}{h(-D)\log y},
	\end{align*}
	for $(Dx)^{1/2+\varepsilon}\leq y\leq x$. Unconditionally, Zaman used his asymptotic formula for the congruence sum \eqref{eq:congruenceSum} and Selberg's sieve to establish a similar Brun-Titchmarsh-type result 
	in short intervals \cite[Theorem 1.4]{Zam}, with the same order of magnitude. 
	
	\begin{theoremA}[Zaman \cite{Zam}] \label{brunzaman}
		Let $f(u,v)=au^2+buv+cv^2$ be a reduced positive definite quadratic form of discriminant $-D$, and let $\varepsilon>0$ be arbitrary. Suppose that
		\begin{align} \label{18_10}
		\bigg(\dfrac{D^{2}}{a}\bigg)^{1/2+\varepsilon}\,x^{1/2+\varepsilon}\leq y\leq x.
		\end{align}
		Then,
		\begin{align*} 
		\pi_f(x)-\pi_f(x-y) < \dfrac{2}{(1-\theta')}\cdot\dfrac{\delta_f\,y}{h(-D)\log y}\bigg(1+O_\varepsilon\bigg(\dfrac{\log\log y}{\log y}\bigg)\bigg),
		\end{align*} 
		where 
		\begin{align*} 
		\theta'=\dfrac{\log x}{2\log y} + \bigg(\dfrac{3}{4}+\dfrac{\varepsilon}{4}\bigg)\dfrac{\log D}{\log y} - \dfrac{\log a}{2\log y}.
		\end{align*} 
	\end{theoremA}
	
	\medskip
	
\noindent Using Theorem \ref{16_43_02_08}, we are able to establish an analogous result to Theorem A, for a range beyond \eqref{18_10}.
	
	\begin{theorem} \label{brun}
		Let $f(u,v)=au^2+buv+cv^2$ be a reduced\,\footnote{\,\,\, The hypothesis of being reduced can be removed, and Theorem \ref{brun} holds for any primitive positive definite quadratic form, by considering $a=1$ in the range \eqref{15_555} and in the values of $\theta_1$ and $\theta_2$. A similar situation occurs in Theorem A (see Remark (ii) in \cite[Theorem 1.4]{Zam}).} positive definite quadratic form of discriminant $-D$. Then, the following statements hold.
		\begin{enumerate}
			\item Let $0<\varepsilon<1/20$ be arbitrary, and suppose that
			\begin{align} \label{15_555}
			\dfrac{D^{2}}{a}\,x^{1/3+\varepsilon}\leq y\leq x^{4/9}.
			\end{align}
			Then,
			$$
			\pi_f(x)-\pi_f(x-y) < \dfrac{4}{(1-\theta_1)}\cdot\dfrac{\delta_f\,y}{h(-D)\log y}\bigg(1+O_\varepsilon\bigg(\dfrac{\log\log y}{\log y}\bigg)\bigg),
			$$
			where 
			$$
			\theta_1 = \dfrac{\log x}{3\log y} + \bigg(\dfrac{4}{3}+\varepsilon\bigg)\dfrac{\log D}{\log y} -  \dfrac{\log a}{\log y}.
			$$
			\item Suppose that
			\begin{align} \label{15_556}
			x^{4/9}\leq y\leq x^{3/5} \hspace{0.3cm} \mbox{and} \hspace{0.3cm} x\geq D^{18}.
			\end{align}
			Then,
			$$
			\pi_f(x)-\pi_f(x-y) < \dfrac{7}{(1-\theta_2)}\cdot\dfrac{\delta_f\,y}{h(-D)\log y}\bigg(1+O\bigg(\dfrac{\log\log y}{\log y}\bigg)\bigg),
			$$where
			$$
			\theta_2 = \dfrac{\log x}{4\log y} +\dfrac{31\log D}{12\log y} - \dfrac{7\log a}{4\log y}.
			$$
		\end{enumerate}
	\end{theorem}
	
	\medskip
	
\noindent	Theorem \ref{brun} states a Brun-Titchmarsh-type inequality in short intervals, for $x^{1/3+\varepsilon}\lesssim y\leq x^{3/5}$, extending the range \eqref{18_10} in Theorem A. This also improves the constant in the range $x^{1/2+\varepsilon}\leq y\leq x^{3/5}$, since we have that
	$$
	\dfrac{7}{1-\theta_2}< \dfrac{2}{1-\theta'}<\dfrac{2}{\varepsilon}.
	$$
	The associated constants in our results can be estimated, uniformly, by
	$$
	16<\dfrac{4}{1-\theta_1}< \dfrac{16}{9\varepsilon}, \hspace{0.3cm} \mbox{and} \hspace{0.3cm}  12<\dfrac{7}{1-\theta_2}\leq \dfrac{672}{11}.$$
	We highlight that, unlike in Theorem A, even under the assumption of GRH, the order of magnitude of the bounds in Theorem \ref{brun} cannot be obtained using \eqref{14_17}.\footnote{\,\,\, Assuming GRH for quadratic Dirichlet $L$-functions modulo $D$, Theorem \ref{brun} can be stated with slight changes in the power of $D$ on the ranges, and in the definition of $\theta_1$ and $\theta_2$. A similar situation occurs in Theorem A (see \cite[Theorem 1.4]{Zam}).} 
	
	\smallskip

	As we shall see, the special case $y=x^{1/2}$ will be useful in the following form:
	
	\begin{corollary} \label{cor:brun} Let $f(u,v)=au^2+buv+cv^2$ be a fixed primitive positive definite quadratic form of discriminant $-D$. Then, 
		$$
		\pi_f(x+\sqrt{x})-\pi_f(x) \leq \dfrac{28\,\delta_f\,\sqrt{x}}{h(-D)\log x}(1+o(1)),
		$$
		as $x\to\infty$.
	\end{corollary}

	\subsection{Cram\'er-type result} Let $\pi(x)$ denote be the number of primes up to $x$. A classical theorem of Cram\'er \cite{cramer} states that, assuming the Riemann Hypothesis (RH), there are constants $c,\alpha>0$ such that
	$$
	\dfrac{\pi(x+c\sqrt{x}\log x)-\pi(x)}{\sqrt{x}}>\alpha
	$$
	for all $x$ sufficiently large. Recently, using a 
	Fourier analysis approach, Carneiro, Milinovich and Soundararajan \cite[Theorem 1.3]{CMS} established this estimate in an optimized explicit form. They proved that, under RH, for $\alpha\geq 0$ we have
	\begin{align*} \inf\bigg\{c>0;\,\, \displaystyle\liminf_{x\to\infty}\,\dfrac{\pi(x+c\sqrt{x}\log x)-\pi(x)}{\sqrt{x}}> \alpha\bigg\}< \dfrac{21}{25}(1+2\alpha).
	\end{align*}
	This was slighlty improved by Chirre, Pereira and de Laat \cite{CPL}, replacing $21/25=0.84$ by $0.8358$. Furthermore, they obtained an analogous result for primes in arithmetic progressions. Our next result extends these techniques for primes represented by quadratic forms.

	\begin{theorem}	\label{thm:cramer} Let $f$ be a primitive positive definite quadratic form of discriminant $-D$. Assume the Generalized Riemann Hypothesis for Hecke $L$-functions. Then, for $\alpha\geq 0$, 
		\begin{align*}
		\inf\bigg\{c>0;\,\, \displaystyle\liminf_{x\to\infty}\dfrac{\pi_f\left(x+c\,\sqrt{x}\log x\right)-\pi_f(x)}{\sqrt{x}}> \alpha\bigg\}< 1.837\,\frac{(\delta_f+\alpha)\,h(-D)}{\delta_f}.
		\end{align*}
	\end{theorem}
	\smallskip
	
\noindent  In particular, for a fixed primitive positive definite quadratic form $f$ of discriminant $-D$, there is always a prime number represented by $f$ in the interval $[x,\, x + 1.837\,h(-D)\sqrt{x}\log x],$ for $x$ sufficiently large. Then, we deduce the following conditional estimate for large gaps between primes represented by a quadratic form.\footnote{\,\,\, To the best of our knowledge, there is no other explicit result of this type in the literature.} 

\begin{corollary}\label{cor:gaps} Let $f$ be a primitive positive definite quadratic form of discriminant $-D$, and let $p_{n,f}$ be the $n$-th prime represented by $f$. Assume the Generalized Riemann Hypothesis for Hecke $L$-functions. Then, 
	\begin{equation} \label{22_17}
	\limsup_{n\to\infty}\dfrac{p_{n+1,f}-p_{n,f}}{\sqrt{p_{n,f}}\,\,\log p_{n,f}}< 1.837\,h(-D).
	\end{equation}
\end{corollary}	

\smallskip

\begin{remark} Consider the quadratic form $f(u,v)=u^2+mv^2$, where $m$ is a positive integer. It is known that there are at most 66 positive integers $m$, such that $f$ represents
a prime $p$ if and only if $p$ belongs to a certain union of arithmetic progressions (see \cite{Cox}). For instance, when $m=1$, a classical theorem of Fermat states that a prime $p$ is represented by $f$, if and only if $p\equiv 1 \,( \mathrm{mod} \, 4)$. In this case, $D=4$, $h(-D)=1$, and we can recover the estimate \eqref{22_17} from \cite[Corollary 2]{CPL}, with the better constant $1.7062$. However, in general, the characterization of such primes is more subtle. For instance, consider the case $m=27$, where $D=108$ and $h(-D)=3$. A conjecture of Euler, proven by Gauss, states that $p$ has the form $u^2+27v^2$ if and only if both $p\equiv 1 \,( \mathrm{mod} \, 3)$, and $2$ is a cubic residue $(\mathrm{mod} \, p)$. This cannot be described by just unions of arithmetic progressions, so the results of \cite{CPL} no longer apply. 
\end{remark} 


	\smallskip
	
	\subsection{Outline of the proof}
	There are two main themes that will be ubiquitous throughout this paper. The first theme is the use of Fourier analysis, in the following way. We begin by finding a summation formula that connects our object of study with an arbitrary function and its Fourier transform. Then, we choose an appropriate test function that recovers the desired information in an optimized manner. The second is the well-known theme that propositions about quadratic forms can be stated in two other equivalent languages: ideals of number fields and lattices. We now discuss the main ideas in each theorem.
	
	
	\subsubsection{Congruence sums}
	The first step is obtaining a summation formula associated with the coefficients $r_f(n)$, relating it to an arbitrary test function and its Fourier transform. These types of formulas are well-known, and are equivalent to the modularity of certain theta series associated with a quadratic form $f$ and a discrete periodic function $\chi$ (see, for instance, \cite[p. 83]{IK} and \cite[p. 32]{123}), the latter which, in this case, allows us to filter out the congruence condition $\ell\,|\,n$. Since we were unable to find an explicit statement in the literature, we provide a proof of the specific summation formula that we require. 
	In Section \ref{sec:congruence}, we obtain the desired expression from an application of the classical Poisson summation formula for the lattice associated with the quadratic form $f$, combined with the discrete Fourier expansion of the periodic function $\chi$. In Section \ref{section3}, we prove Theorem \ref{16_43_02_08} following an approach outlined in \cite[Section 4.4]{IK}, which was applied to the Gauss circle problem in \cite[Corollary 4.9]{IK}. By choosing an appropriate test function in our summation formula and carrying out an asymptotic analysis (for instance, see Lemma \ref{23_54} in Appendix), we arrive at our new estimate for \eqref{eq:congruenceSum}. We highlight that a good explicit dependence on $\ell$ and the parameters of $f$ is required. This imposes significant technical difficulties when compared to the argument in \cite{IK}, and requires a careful analysis and delicate manipulations with a reduced quadratic form. 
	
	\begin{remark}
		Higher moments of $r_f(n)$ have also been studied by Blomer and Granville \cite{BG}. Later, Xu \cite{Xu} gave some improvements in their error terms. Additionally, he proved that, when $\ell=1$ in Theorem \ref{16_43_02_08}, the optimal error term in \eqref{15_56} satisfies $\Omega(D^{1/4}x^{1/4})$, which generalizes the classical omega result given originally by Hardy and Landau (see \cite{grosswald}).
	\end{remark}

	\subsubsection{Brun-Titchmarsh-type result} In Section \ref{sec:brun} we prove Theorem \ref{brun}, following Zaman's general outline in \cite{Zam}. Here the main strategy is an application of Selberg's sieve \cite[Theorem 7.1]{opera}, which transforms the problem of obtaining an upper bound for primes represented by $f$ in short intervals, into the problem of estimating the associated congruence sums \eqref{eq:congruenceSum}. 
	We remark that our extended range in Theorem \ref{brun} comes from the improved error term in our estimate of the congruence sums \eqref{eq:congruenceSum}, of the form $O_{f,\,\ell}(x^{1/3})$, given in Theorem \ref{16_43_02_08}. When $x$ is large compared to $\ell$, this improves the estimate $O_{f,\,\ell}(x^{1/2})$ given in \cite[Proposition 7.1]{Zam}, and it allows us to take intervals around $x$ of size as small as roughly $x^{1/3}$.
	
	\subsubsection{Cram\'er-type result}
We follow the argument of Carneiro, Milinovich, and Soundararajan in \cite[Section 5]{CMS}, to prove Theorem \ref{thm:cramer}. Here, we work with the language of ideals in imaginary quadratic fields. This allows us to use the machinery of Hecke characters and Hecke $L$-functions to obtain information about prime ideals in a given ideal class, and therefore, about prime numbers represented by a given quadratic form $f$. We first give some necessary background on Hecke $L$-functions, and their relation to quadratic forms, in Section \ref{sec:hecke}. The main ingredients in Theorem \ref{thm:cramer} are our version of the Brun-Titchmarsh inequality in Corollary \ref{cor:brun}, and the Guinand-Weil explicit formula for $L$-functions (see, for instance, \cite[Theorem 5.12]{IK} and \cite[Lemma 5]{CF}). 
	Then, we establish a version for Hecke $L$-functions that averages over all Hecke characters in a given congruence class group. We finish the proof of Theorem \ref{thm:cramer} in Section \ref{sec:cramer}. Following \cite{CMS}, we start with an arbitrary function $F$ in our version of the Guinand-Weil formula. The strategy then consists of taking a suitable dilation and modulation of $F$, so that we emphasize, in our explicit formula, intervals containing few prime numbers represented by $f$. We must then carry out an asymptotic analysis, and choose an appropriate function $F$ at the end, to conclude the desired result. In Section \ref{sec:uncertainty}, we discuss some qualitative aspects of the problem of choosing an optimal function $F,$ related to the uncertainty principle. 
	
	\smallskip 
	
	\subsection{Remarks and Notation} Let $f(u,v)=au^2+buv+cv^2$ be a positive definite quadratic form of discriminant $-D$, and without loss of generality assume that $a,\, c\geq 1$. In the case when $f$ is reduced, since $|b|\leq a\leq c$, we have that $a\ll \sqrt{D}$ and $D\geq 3$. We will use these frequently. Moreover, we have that $r_f(0)=1$, and $a$ is the smallest positive integer represented by $f$.
	
	\noindent The symbols $\ll$, $O(\,\cdot\,)$, $o(\,\cdot\,)$, and $\asymp$ are used in the standard way. In the subscript, we indicate the parameters on which the implicit constant may depend. We also denote $x_+:=\max\{x,0\}$. For a function $G\in L^1(\R^n)$, we define its Fourier transform by $$\widehat{G}(\xi)=\int_{\R^n}G(y)e^{-2\pi i\,\xi\,\cdot\, y}\d y.$$ For a radial function $G:\R^n\to\R$, we use the notation $G(x)=G(|x|)$. 
	
	\medskip
	
	\section{Summation formula for $r_f(n)$}\label{sec:congruence}
	
	Let $f(u,v)=au^2+buv+cv^2$ be a positive definite quadratic form of discriminant $-D$.  We recall that, for $n\geq 0$,
	\begin{align*} 
	r_f(n)=\#\left\{(u,v)\in \Z^2:f(u,v)=n\right\}.
	\end{align*} 
	
	\begin{lemma}\label{thm:poisson} 
		Let $G\in L^1(\R^2)$ be a radial continuous function. Suppose that
		\begin{align} \label{15_37}
		|G(x)|\ll \dfrac{1}{(1+|x|^2)^{1+\delta}}\,\,\,\mbox{and}\,\,\,\,\,\,\, |\widehat{G}(\xi)|\ll \dfrac{1}{(1+|\xi|^2)^{1+\delta}},
		\end{align}
		for some $\delta>0$. Then, for an integer $\ell\geq 1$, we have
		\begin{align}  \label{00_01}
		\begin{split}
		\sum_{\substack{n=0\\ \ell|n}}^\infty r_f(n)\,G(\sqrt{n}) = \,\,& \dfrac{2\,\widetilde{g}(\ell)}{\sqrt{D}}\,\sum_{n=0}^\infty r_f(n)\,\widehat{G}\bigg(\sqrt{\frac{4n}{D}}\bigg) \\
		&  +  O\Bigg(\frac{\widetilde{g}(\ell)\,\ell^2}{\sqrt{D}}\max_{\substack{0\le r,\,s<\ell\\ (r,s)\in \Z^2\setminus (0,0)}}
		\sum_{(u, v)\in \Z^2} \left|\widehat{G}\left(\sqrt{\frac{4f(u-r/\ell, v-s/\ell)}{D}}\right)\right|\Bigg),
		\end{split} 
		\end{align}
		where
		\begin{align} \label{17_34}
		\widetilde{g}(\ell)=\dfrac{1}{\ell^2}\,\#\left\{(u,v)\in \Z^2:0\leq u,v<\ell, \hspace{0.06cm} \mbox{and} \hspace{0.2cm} \ell\hspace{0.05cm}|\hspace{0.05cm}f(u,v)\right\}.
		\end{align}
		
	\end{lemma}
	\begin{proof}
		We start associating a lattice $\Lambda\subset\R^2$, defined by the basis $\{\omega_1, \omega_2\}$, to the quadratic form $f$ in such a way that $a=|\omega_1|^2$, $b=2\,\omega_1 \cdot \omega_2$ and $c=|\omega_2|^2$. This implies that
		\begin{align} \label{14_36}
		|u\,\omega_1+v\,\omega_2|^2=f(u,v), \hspace{0.2cm} \mbox{and}\hspace{0.2cm}  r_f(n)=\#\{{\omega}\in\Lambda: |\omega|^2=n\},\hspace{0.1cm}  \mbox{for} \hspace{0.1cm} n\geq 0.
		\end{align}
		Let us consider the abelian group\footnote{\,\,\, We recall that the set $\Lambda/\ell\Lambda$ is defined by the equivalence classes in $\Lambda$ given by
			$\overline{\omega}=\{\omega + \ell\lambda: \lambda\in\Lambda\}$.} $(\Lambda/\ell\Lambda,+)$ of order $\ell^2$, and let $\chi:\Lambda/\ell\Lambda\to \C$ be the function defined by
		\begin{equation*}\chi(\overline{\omega}) = \left\{
		\begin{array}{ll}
		1,      & \mathrm{if\ } \ell\hspace{0.08cm}|\hspace{0.08cm}|\omega|^2; \\
		0     & \mathrm{otherwise.\ }
		\end{array}
		\right.
		\end{equation*}
		Then, since $G$ satisfies \eqref{15_37}, we have 
		\begin{align} \label{4}
		\sum_{\substack{{\omega}\in \Lambda  \vspace{0.07cm}\\ \ell\hspace{0.03cm}|\hspace{0.03cm}|\omega|^2 }}G({\omega})=\sum_{{\omega}\in \Lambda}\chi(\overline{\omega})G({\omega}).
		\end{align} 
		On the other hand, we consider the dual lattice $\Lambda^*=\{{\omega}^*\in \R^2: {\omega}\cdot{\omega}^* \in \Z, \hspace{0.05cm}\mbox{for all} \hspace{0.1cm} {\omega}\in\Lambda\}$. It has a basis $\{{\omega^*_1}, {\omega^*_2}\}$, given by $
		{\omega^*_1} = 4c\,{\omega_1}/D - 2b\,{\omega_2}/D$ and
		$
		{\omega^*_2} = {2b}\,{\omega_1}/D - 4a\,{\omega_2}/D.
		$
		This implies that \begin{align} \label{14_37}
		|u\,{\omega^*_1}+v\,{\omega^*_2}|^2=4f(v,u)/D, \hspace{0.2cm} \mbox{and}\hspace{0.2cm}  r_f(n)=\#\{{\omega}^*\in\Lambda^*: |\omega^*|^2=4n/D\},\hspace{0.1cm}  \mbox{for} \hspace{0.1cm} n\geq 0.
		\end{align}
		For each $\lambda^*$ in the set $P=\{sw_1^*+rw_2^*: 0\leq s,r<\ell \,\,\mbox{and}\,\, s,r\in\Z\}$, we define a character $e_{\lambda^*}$ in the group $(\Lambda/\ell\Lambda,+)$ by $e_{\lambda^*}(\overline{w})=e^{2\pi i\,\omega\,\cdot\,\lambda^*/\ell}$. Since the cardinality of $P$ is $\ell^2$, we conclude that $\{e_{\lambda^*}\}_{\lambda^*\in P}$ are all the characters in the group $(\Lambda/\ell\Lambda,+)$ (see \cite[Theorem 2.5 in Chapter 7]{SS}). Now, we define the Fourier coefficient of $\chi$ with respect to $e_{\lambda^*}$, by
		$$
		\widehat{\chi}(e_{\lambda^*})=\dfrac{1}{\ell^2}\sum_{\overline{\omega}\in \Lambda/\ell\Lambda}\chi(\overline{\omega})\,e^{-2\pi i\,\omega\,\cdot\,\lambda^*/\ell}.
		$$
		Then, the Fourier inversion formula (see \cite[Theorem 2.7 in Chapter 7]{SS}) yields
		\begin{align} \label{18_45} 
		\chi(\overline{w})=\sum_{\lambda^*\in P}\widehat{\chi}(e_{\lambda^*})\,e^{2\pi i\,\omega\,\cdot\,\lambda^*/\ell}.
		\end{align}
		Combining \eqref{4}, \eqref{18_45} and Fubini's theorem, we get
		\begin{align*} 
		\sum_{\substack{\omega\in\Lambda\\ \ell\hspace{0.03cm}|\hspace{0.03cm}|\omega|^2}} G(\omega) =\sum_{\lambda^*\in P}\widehat{\chi}(e_{\lambda^*})\Bigg(\sum_{w\in\Lambda} G(w)\,e^{2\pi i\,\omega\,\cdot\,\lambda^*/\ell}\Bigg).
		\end{align*}
		Recalling that $vol(\Lambda^*)=\sqrt{4/D}$, we use the Poisson summation formula for lattices in the above inner sum (since $G$ satisfies \eqref{15_37}) to find that
		\begin{align} \label{17_43}
		\sum_{\substack{\omega\in\Lambda\\ \ell\hspace{0.03cm}|\hspace{0.03cm}|\omega|^2}} G(\omega) =\sqrt{\dfrac{4}{D}}\sum_{\lambda^*\in P}\widehat{\chi}(e_{\lambda^*})\sum_{w^*\in\Lambda^*} \widehat{G}\bigg(w^*-\dfrac{\lambda^*}{\ell}\bigg).
		\end{align}
		On the other hand, if we define
		\begin{align*} 
		\widetilde{g}(\ell)=\dfrac{1}{\ell^2}\,\#\left\{\overline{\omega}\in \Lambda/\ell\Lambda:  \ell\hspace{0.08cm}|\hspace{0.08cm}|\omega|^2\right\},
		\end{align*}
		it is clear that $\widehat{\chi}(e_{0^*})= \widetilde{g}(\ell)$ and $|\widehat{\chi}(e_{\lambda^*})|\leq \widetilde{g}(\ell)$. Therefore, isolating the point $\lambda^*=0$ in \eqref{17_43} gives us
		\begin{align*} 
		\sum_{\substack{\omega\in\Lambda\\ \ell\hspace{0.03cm}|\hspace{0.03cm}|\omega|^2}} G(\omega) & = \frac{2\,\widetilde{g}(\ell)}{\sqrt{D}}\sum_{\omega^*\in\Lambda^*} \widehat{G}(\omega^*) + \frac{2}{\sqrt{D}}\sum_{\lambda^*\in P\setminus\{0^*\}}\widehat{\chi}(e_{\lambda^*})\sum_{w^*\in\Lambda^*} \widehat{G}\bigg(w^*-\dfrac{\lambda^*}{\ell}\bigg) \\
		& =\frac{2\,\widetilde{g}(\ell)}{\sqrt{D}}\sum_{n=0}^\infty r_f(n)\,\widehat{G}\bigg(\sqrt{\frac{4n}{D}}\bigg) + O\left(\frac{\widetilde{g}(\ell)\,\ell^2}{\sqrt{D}}\,\max_{\substack{0\le r,\,s<\ell\\ (r,s)\in \Z^2\setminus (0,0)}}
		\sum_{(u, v)\in \Z^2} \left|\widehat{G}\left(\sqrt{\frac{4f(u-r/\ell, v-s/\ell)}{D}}\right)\right|\right),
		\end{align*}
		where we have used \eqref{14_37} and the fact that $\widehat{G}$ is radial. We conclude the proof using \eqref{14_36}.
	\end{proof}

\noindent The following technical lemma will help us estimate the error term. We compare a small translation of $f$ with the untranslated value, outside of a finite number of exceptions.
	\begin{lemma}\label{lem:translations}
		Suppose that $f$ is reduced. Let $\ell, r, s$ be integers such that $\ell\ge 1$ and $0\le r,s < \ell$. Then,
		\[\# \bigg\{(u,v)\in \Z^2: f\left(u-\frac{r}{\ell},v-\frac{s}{\ell}\bigg)<\frac{f(u,v)}{2}\right\}\ll \frac{\sqrt{D}}{a}.
		\]
	\end{lemma}
	\begin{proof} Define the set
		$$
		A=\bigg\{(u,v)\in \Z^2: f\bigg(u-\frac{2r}{\ell},v-\frac{2s}{\ell}\bigg)< 6c
		\bigg\}.
		$$
		First we show that
		\begin{align} \label{1_15}
		\bigg\{(u,v)\in \Z^2: f\bigg(u-\frac{r}{\ell},v-\frac{s}{\ell}\bigg)<\frac{f(u,v)}{2}
		\bigg\}\subset A.
		\end{align}
		Indeed, if $(u,v)\in A^c$, using that $f$ is reduced, we have
		\begin{align}\label{eq:11_8_11_34}
		\frac{2\,f(r,s)}{\ell^2} \le 6c \le f\left(u-\frac{2r}{\ell},v-\frac{2s}{\ell}\right).
		\end{align}
		Applying the identity 
		\begin{equation}\label{eq:11_8_11_40}
		f(u-x,v-y)= f(u,v) + f(x,y) - 2aux - buy - bxv -2cvy
		\end{equation}
		in \eqref{eq:11_8_11_34} yields
		\begin{align*}
		&\frac{2\,f(r,s)}{\ell^2} \le f(u,v) + \frac{4\,f(r,s)}{\ell^2} -\frac{4aur+2bus+2bvr+4cvs}{\ell}.
		\end{align*}
		Then,
		$$-\frac{f(r,s)}{\ell^2}+ \frac{2aur+bus+bvr+2cvs}{\ell} \le \frac{f(u,v)}{2}.
		$$	
		Using this inequality and identity \eqref{eq:11_8_11_40}, we see that
		\begin{align*}
		f\bigg(u-\frac{r}{\ell}, v-\frac{s}{\ell}\bigg) &= f(u,v) + \frac{f(r,s)}{\ell^2} -\frac{2aur+bus+bvr+2cvs}{\ell} \ge \frac{f(u,v)}{2}. 
		\end{align*}
		This shows \eqref{1_15}, and it now suffices to obtain an upper bound for the cardinality of $A$. Observe that
		\begin{align*}
		\#A & = \# \{(u,v)\in \Z^2: f(u\ell-2r,v\ell-2s)<6c\ell^2\}\\
		& \leq \#\{(u,v)\in\Z^2:\, f(u,v)\le 6c\ell^2,\, u\equiv -2r \,(\mathrm{mod}\, \ell), \, v\equiv -2s\,(\mathrm{mod}\, \ell)
		\}.
		\end{align*}
		We now proceed with the well-known argument in \cite[Lemma 3.1]{BG} as follows.
		Rewriting $f(u,v)$, we must bound the number of integer solutions to the inequality $(2au+bv)^2 + Dv^2\le 24ac\ell^2.$ A solution $(u,v)$ must satisfy that $|v|\ll \ell$ (where we used that $ac\ll D),$ and that 
		\[ \frac{-\sqrt{24ac\ell^2-Dv^2}-bv}{2a} \le u \le \frac{\sqrt{24ac\ell^2-Dv^2}-bv}{2a}.
		\]
		Therefore, $v$ belongs to an interval of size at most $\ll \ell$, and $u$ belongs to an interval of size at most $\ll \sqrt{D}\,\ell/a$ (once again using that $ac\ll D$). Hence, the number of solutions $(u,v)$ with the desired congruences modulo $\ell$ is at most $\ll \sqrt{D}/a$.
	\end{proof}

	\section{Proof of Theorem \ref{16_43_02_08}} \label{section3}
	In \cite[Proposition 7.1]{Zam}, Zaman used a lattice point counting argument, via geometry of numbers methods, to estimate \eqref{eq:congruenceSum}. He established the following: for a primitive positive definite quadratic form $f$ of discriminant $-D$, and a squarefree integer $\ell\geq 1$, we have 
	\begin{align} \label{15_27}
	\sum_{\substack{n\le x\\ \ell|n}} r_f(n) = \frac{2\pi}{\sqrt{D}}\,g(\ell)\,x + 
	O\left(\frac{\tau_3(\ell)\,a^{1/2}}{D^{1/2}}x^{1/2} + \frac{\tau(\ell)\,\tau_3(\ell)\,\ell^{1/2}\,D^{1/4}}{a^{3/4}}x^{1/4}+1 \right),
	\end{align}
	for $x\geq 1$. Here, $g$ is a multiplicative function satisfying 
	\begin{align}\label{15_04}
	g(p)=\dfrac{1}{p}\bigg(1+\chi(p)-\dfrac{\chi(p)}{p}\bigg)
	\end{align}
	for all primes $p$, $\chi=\chi_{-D}$ is the corresponding Kronecker symbol, and $\tau_3$ is the $3$-divisor function. The main goal here is to improve the error term in \eqref{15_27}, reducing $x^{1/2}$ to $x^{1/3}$.
	
	\subsection{Proof of Theorem \ref{16_43_02_08}} We partially follow the approach outlined in \cite[Corollary 4.9]{IK}. Assume that $x\geq 1$ is a real number and $\ell\geq 1$ is an integer. Let $1\leq y\leq x^{1/2}$ be a parameter to be chosen. 
	We will apply Lemma \ref{thm:poisson} to the radial function $G:\R^2\to\R$ supported in $0\leq r\leq (x+y)^{1/2}$, and defined by
	$$
	G_{x,y}(r)=G(r):=\min\bigg\{r^2,1,\dfrac{x+y-r^2}{y}\bigg\}.
	$$
	By Lemma \ref{23_54}, the function $G$ satisfies the conditions \eqref{15_37}, with the bounds
	\begin{equation} \label{eq:H111}
	\big|\widehat{G}(\sqrt{\xi})\big| \ll 
	\dfrac{x^{1/4}}{|\xi|^{3/4}} \hspace{0.2cm} \mbox{for} \hspace{0.2cm} {|\xi|\neq 0},  \hspace{0.2cm} \mbox{and}  \hspace{0.2cm}  \big|\widehat{G}(\sqrt{\xi})\big| \ll \dfrac{x^{3/4}}{y|\xi|^{5/4}} \hspace{0.2cm} \mbox{for} \hspace{0.2cm} {|\xi|\geq 1}.
	\end{equation}
	Now, let us analyze the right-hand side of \eqref{00_01}.  We recall that $r_f(0)=1$, and by Lemma \ref{23_54}, we know that $\widehat{G}(0)=\pi x + O(y)$. Letting $z=Dx/y^2$ (note that $4z/D\geq 1$), and using the estimates in \eqref{eq:H111} we obtain
	\begin{align*} 
	\Bigg|\sum_{n=1}^\infty r_f(n)\,\widehat{G}\bigg(\sqrt{\frac{4n}{D}}\bigg)\Bigg| & = \Bigg|\sum_{a\leq n\leq z} r_f(n)\,\widehat{G}\bigg(\sqrt{\frac{4n}{D}}\bigg)+\sum_{n>z} r_f(n)\,\widehat{G}\bigg(\sqrt{\frac{4n}{D}}\bigg)\Bigg| \\
	& \ll D^{3/4}x^{1/4}\sum_{a\leq n\leq z}\dfrac{r_f(n)}{n^{3/4}} +\dfrac{D^{5/4}x^{3/4}}{y}\sum_{n>z} \dfrac{r_f(n)}{n^{5/4}}.
	\end{align*}
	To estimate the sums above, we use integration by parts and the well-known result (see \cite[Lemma 3.1]{BG})
	\begin{equation*} 
	\sum_{a\leq n\leq x}r_f(n)= \frac{2\pi x}{\sqrt{D}} + O\left(\sqrt{\frac{x}{a}}\right),
	\end{equation*}
	for $x\geq a$. Therefore,
	\begin{align*}  
	\Bigg|\sum_{n=1}^\infty r_f(n)\,\widehat{G}\bigg(\sqrt{\frac{4n}{D}}\bigg)\Bigg| &  \ll \dfrac{D^{3/4}x^{1/4}}{a^{3/4}}+ \dfrac{D^{1/2}x^{1/2}}{y^{1/2}}.
	\end{align*}
	We now estimate the translated terms in \eqref{00_01}. Let $r,\, s$ be integers such that $0\le r,\, s<\ell$ and $(r,s)\neq (0,0)$. Let
	\[B:=\bigg\{(u,v)\in \Z^2: f\bigg(u-\frac{r}{\ell},v-\frac{s}{\ell}\bigg)<\frac{f(u,v)}{2}
	\bigg\}\cup\{(0,0)\}
	\]
	be the set in the statement of Lemma \ref{lem:translations} (with the point $(0,0)$ included). First, let us bound the sum over $(u,v)\in B$. We will use the fact that $f(u-r/\ell, v-s/\ell)\ge a/{\ell^2}$ for all $(u,v)\in \Z^2$, and Lemma \ref{lem:translations}. Then, recalling that $a\ll D^{1/2}$ and using \eqref{eq:H111}, we see that
	\begin{align*}
	\sum_{(u,\, v)\in B} \left|\widehat{G}\left(\sqrt{\frac{4f(u-r/\ell, v-s/\ell)}{D}}\right)\right|
	&\ll (\#B) \max_{(u,v)\in \Z^2}\left\{\frac{x^{1/4}D^{3/4}}{f(u-r/\ell, v-s/\ell)^{3/4}}\right\}\ll \frac{\ell^{3/2}D^{5/4}x^{1/4}}{a^{7/4}}.
	\end{align*}
	We analyze the sum over $(u,v)\in B^c$, by splitting it once more into the sets ${B^c\cap \{f(u,v)\leq z\}}$ and ${ B^c\cap \{f(u,v)>z\}}$. We estimate it using \eqref{eq:H111} as follows:
	\begin{align*}
	&\sum_{(u,\, v)\in B^c} \left|\widehat{G}\left(\sqrt{\frac{4f(u-r/\ell, v-s/\ell)}{D}}\right)\right|
	\\
	&\,\,\,\, \,\,\,\, \ll  \sum_{(u,\, v)\in B^c\cap \{f(u,v)\le z\}}
	\frac{D^{3/4}x^{1/4}}{f(u-r/\ell, v-s/\ell)^{3/4}}	+ \sum_{(u,\, v)\in  B^c\cap \{f(u,v)>z\}} \frac{D^{5/4}x^{3/4}}{y\,f(u-r/\ell, v-s/\ell)^{5/4}}\\
	&\,\,\,\, \,\,\,\, \ll D^{3/4}x^{1/4}\sum_{(u,\, v)\in  B^c\cap\{f(u,v)\le z\}}
	\frac{1}{f(u, v)^{3/4}} +\dfrac{D^{5/4}x^{3/4}}{y} \sum_{\{f(u,v)>z\}} \frac{1}{f(u, v)^{5/4}}\\
	&\,\,\,\, \,\,\,\, \ll D^{3/4}x^{1/4}\sum_{a\leq n\leq z}
	\frac{r_f(n)}{n^{3/4}} +\dfrac{D^{5/4}x^{3/4}}{y} \sum_{n>z} \frac{r_f(n)}{n^{5/4}}\ll  \dfrac{D^{3/4}x^{1/4}}{a^{3/4}}+ \dfrac{D^{1/2}x^{1/2}}{y^{1/2}}.
	\end{align*}
	Therefore, since $G(0)=0$, we combine all the terms in \eqref{00_01} to find, for $1\leq y\leq x^{1/2}$,
	\begin{align}  \label{1_16}
	\begin{split} 
	\sum_{\substack{n=1\\ \ell|n}}^\infty r_f(n)\,G_{x,y}(\sqrt{n}) = \frac{2\pi}{\sqrt{D}} \,\widetilde{g}(\ell)\,x+ O\Bigg(\widetilde{g}(\ell)\bigg(\frac{\ell^{7/2}D^{3/4}x^{1/4}}{a^{7/4}}+ \dfrac{\ell^2x^{1/2}}{y^{1/2}}+\dfrac{y}{D^{1/2}}\bigg)\Bigg),
	\end{split}
	\end{align}
	where $\widetilde{g}(\ell)$ was defined in \eqref{17_34}. Since $G_{x,y}(r)\geq 0$, we truncate the sum on the left-hand side of \eqref{1_16} over $1\leq n\leq x$. Using the definition of $G$, this implies that 
	\begin{align}  \label{1_34}
	\sum_{\substack{1\leq n \leq x\\ \ell|n}} r_f(n)\leq  \frac{2\pi}{\sqrt{D}} \,\widetilde{g}(\ell)\,x+ O\Bigg(\widetilde{g}(\ell)\bigg(\frac{\ell^{7/2}D^{3/4}x^{1/4}}{a^{7/4}}+ \dfrac{\ell^2x^{1/2}}{y^{1/2}}+\dfrac{y}{D^{1/2}}\bigg)\Bigg).
	\end{align}
	To obtain the inverse inequality, we replace $x$ by $x-y$ (in this case $1\leq y\leq (x-y)^{1/2}$) in \eqref{1_16} and use the fact that 
	$$
	\sum_{\substack{n=1\\ \ell|n}}^\infty r_f(n)\,G_{x-y,y}(\sqrt{n}) = \sum_{\substack{1\leq n\leq x\\ \ell|n}} r_f(n)\,G_{x-y,y}(\sqrt{n}) \leq \sum_{\substack{1\leq n\leq x\\ \ell|n}} r_f(n).
	$$
	This yields
	\begin{align} \label{1_311}
	\begin{split} 
	\sum_{\substack{1\leq n\leq x\\ \ell|n}} r_f(n)\geq  \frac{2\pi}{\sqrt{D}} \,\widetilde{g}(\ell)\,x+ O\Bigg(\widetilde{g}(\ell)\bigg(\frac{\ell^{7/2}D^{3/4}(x-y)^{1/4}}{a^{7/4}}+ \dfrac{\ell^2(x-y)^{1/2}}{y^{1/2}}+\dfrac{y}{D^{1/2}}\bigg)\Bigg).
	\end{split}
	\end{align}
	Then, choosing $y=D^{1/3}x^{1/3}/2^{1/2}$ in \eqref{1_34} and \eqref{1_311}, we conclude\footnote{\,\,\, Note that, so far, $\ell\geq 1$ is not necessarily a squarefree integer. Using the estimate $|\widetilde{g}(\ell)|\leq 1$, we obtain a general version of Theorem \ref{16_43_02_08}.} that, for $x\geq D^2$
	\begin{align}  \label{15_342}
	\begin{split} 
	\sum_{\substack{1\leq n\leq x\\ \ell|n}}r_f(n)= \frac{2\pi}{\sqrt{D}} \,\widetilde{g}(\ell)\,x+ O\Bigg(\dfrac{\widetilde{g}(\ell)\,\ell^{2}x^{1/3}}{D^{1/6}}+\frac{\widetilde{g}(\ell)\,\ell^{7/2}D^{3/4}x^{1/4}}{a^{7/4}}\Bigg).
	\end{split}
	\end{align}
	Now, if we compare the main terms in \eqref{15_27} and \eqref{15_342}, we plainly see that $\widetilde{g}(\ell)=g(\ell)$ for any $\ell$ squarefree integer. Also note that, for each prime $p$, \eqref{15_04} implies that $|g(p)|\leq 2/p$. Since $g$ is a multiplicative function, for a squarefree integer $\ell=p_1\ldots p_k$, we have 
	\begin{align*}  
	|\widetilde{g}(\ell)|=|g(\ell)|=|g(p_1\ldots p_k)|\leq \dfrac{2^{k}}{p_1\ldots p_k}=\dfrac{\tau(\ell)}{\ell}.
	\end{align*}
	Inserting this estimate in the error term of \eqref{15_342}, we conclude.
	\qed
	

	\begin{remark}
		We highlight that the asymptotic formula \eqref{15_56} in Theorem \ref{16_43_02_08} holds for $x\geq D^2$. We can establish a similar result for $x\geq 3$, if we choose $y=x^{1/3}/2^{1/2}$ in the previous proof. Then, for $x\geq 3$,
		\begin{align*} 
		\sum_{\substack{1\leq n\leq x\\ \ell|n}}r_f(n)= \frac{2\pi}{\sqrt{D}} \,g(\ell)\,x+ O\Bigg(\tau(\ell)\,\ell\,x^{1/3}+\frac{\tau(\ell)\,\ell^{5/2}D^{3/4}}{a^{7/4}}x^{1/4}\Bigg).
		\end{align*}
		Also, note that the above formula can be extended to any primitive positive definite quadratic form, not necessarily reduced, by considering $a=1$ in the error term.
	\end{remark}
	
	\medskip

	\section{Proof of Theorem \ref{brun}}\label{sec:brun} 
	Let $f(u,v)=au^2+buv+cv^2$ be a reduced positive definite quadratic form of discriminant $-D$, and fix $0<\varepsilon<1/20$. To prove Theorem \ref{brun}, we follow the idea developed in \cite{Zam}. Let $\chi=\chi_{-D}(\cdot):=\big(\frac{-D}{\cdot}\big)$ denote the corresponding Kronecker symbol, which is a quadratic Dirichlet character, and let $L(s,\chi)$ be the associated $L$-function. We remark that, in the ranges \eqref{15_555} and \eqref{15_556}, using the fact that $f$ is reduced, we have that $x\geq x-y\geq D^2$.
	
	Let us define $w=\#\{(p,q,r,s)\in \Z^4: ps-qr=1 \hspace{0.1cm} \mbox{and} \hspace{0.1cm} f(u,v)= f(pu+qv,ru+sv)\}.$
	By \cite[p. 63 Satz 2]{Zag}, we have that $w=6$ when $D=3$, $w=4$ when $D=4$, and $w=2$ otherwise. This implies that, if $p$ is a prime represented by $f$, then it is represented with multiplicity $\delta_f^{-1}w$, where $\delta_f$ is defined in \eqref{20_39}. The number $w$ is related to the class number $h(-D)$ through the class number formula (see \cite[p. 72, Staz 5]{Zag}):
	\begin{align} \label{8.4}
	h(-D)=\dfrac{w\sqrt{D}}{2\pi}L(1,\chi).
	\end{align} 
	We start by dividing into cases, depending on the size of $L(1,\chi)$.
\subsection{The case $L(1,\chi)\geq (\log y)^{-2}$} \label{17_04} Let $z\geq 2$ be a parameter to be chosen later, and define $P=\prod_{p\leq z}p$. Then, one can see that
	\begin{align} \label{8.6}
	\dfrac{w}{\delta_f}\big(\pi_f(x)-\pi_f(x-y)\big)\leq \displaystyle\sum_{\substack{x-y<n\leq x\\(n,P)=1}}r_f(n) + \dfrac{w}{\delta_f}\pi(z).
	\end{align} 
	Let us bound the sieved sum on the right-hand side of \eqref{8.6}. For a squarefree integer $\ell\geq 1$, Theorem \ref{16_43_02_08} gives us
	\begin{align}  \label{18_56}
	\sum_{\substack{x-y< n\le x\\ \ell|n}} r_f(n) = \frac{2\pi \,y}{\sqrt{D}}\,g(\ell) + E_\ell,
	\end{align}
	where
	\begin{align} \label{15_34}
	|E_\ell| \ll \dfrac{\tau(\ell)\,\ell}{D^{1/6}}x^{1/3}+\frac{\tau(\ell)\,\ell^{5/2}D^{3/4}}{a^{7/4}}x^{1/4}.
	\end{align}
	Then, \eqref{18_56} and a direct application of Selberg's upper bound sieve (see \cite[Theorem 7.1 and Eq. (7.32)]{opera} with level of distribution $z^2$ give us
	\begin{align} \label{8.71}
	\displaystyle\sum_{\substack{x-y<n\leq x\\(n,P)=1}}r_f(n)  \leq  \dfrac{2\pi\, y}{\sqrt{D}}\,J^{-1}+ \displaystyle\sum_{\substack{\ell | P\\\ell<z^2}}\tau_3(\ell)|E_\ell|,
	\end{align}
	where $
	J=\sum_{\ell<z, \, \ell | P} h(\ell)$, and $h$ is a multiplicative function defined by 
	$$
	h(\ell)=\prod_{p|\ell}\dfrac{g(p)}{1-g(p)}.
	$$
	To bound the main term in \eqref{8.71}, we treat $g$ as a completely multiplicative function, to obtain (see \cite[Eq. (8.8)]{Zam})
	\begin{align} \label{8.8}
	J\geq \displaystyle\sum_{\ell<z}g(\ell). 
	\end{align}
	To bound the sum on the right-hand side of \eqref{8.71}, we use \eqref{15_34}, and integration by parts with the estimate (see \cite[Theorem 1]{florian})
	$$
	\sum_{n\leq x}\tau_3(n)\tau(n)\ll x(\log x)^5.
	$$
	It follows that
	\begin{align} \label{23_21}  
	\displaystyle\sum_{\substack{\ell | P\\\ell<z^2}}\tau_3(\ell)|E_\ell| &\ll \dfrac{x^{1/3}}{D^{1/6}}\displaystyle\sum_{\ell<z^2}\tau_3(\ell)\tau(\ell)\,\ell+\dfrac{D^{3/4}x^{1/4}}{a^{7/4}}\sum_{\ell<z^2}\tau_3(\ell)\tau(\ell)\,\ell^{5/2} \nonumber \\
	&  \ll \dfrac{x^{1/3}z^4 (\log z)^5}{D^{1/6}} + \dfrac{D^{3/4}x^{1/4}z^7 (\log z)^5}{a^{7/4}}.
	\end{align}
	We now combine \eqref{8.6}, \eqref{8.71}, \eqref{8.8}, \eqref{23_21}, the prime number theorem, and the fact that $x\geq D^{2}$. We obtain
	\begin{align} \label{23_27}
	\dfrac{w}{\delta_f}\big(\pi_f(x)-\pi_f(x-y)\big)\leq \dfrac{2\pi \,y}{\sqrt{D}\displaystyle\sum_{\ell<z}g(\ell)} + O\bigg(\dfrac{x^{1/3}z^4 (\log z)^5}{D^{1/6}} + \dfrac{D^{3/4}x^{1/4}z^7 (\log z)^5}{a^{7/4}}\bigg).
	\end{align} 
	To analyze the main term in the right-hand side of \eqref{23_27}, we use some bounds given in \cite{Zam}. Combining Lemma 4.3, Lemma 4.4 and Lemma 8.2 of \cite{Zam}, for any fixed $0<\varepsilon<1/20$, it follows that 
	\begin{align}  \label{2_582}
	\displaystyle\sum_{\ell<z}g(\ell) \geq L(1,\chi)\log z -\bigg(\dfrac{1}{8}+\varepsilon\bigg)L(1,\chi)\log D + O\big(L(1,\chi)+z^{-\varepsilon^2/2}\big),
	\end{align}
	for any $z\geq 1$ such that $z\gg D^{1/4+\varepsilon}$.
	
	\subsubsection{The first range \eqref{15_555}} We recall that, in the range
	\begin{align*} 
	\dfrac{D^{2}}{a}\,x^{1/3+\varepsilon}\leq y\leq x^{4/9},
	\end{align*}
	we have $x\geq D^{18}/{a^{9}}\geq D^{13}$, since $f$ is reduced. Now, we choose
	\begin{align*} 
	z = \dfrac{a^{1/4} y^{1/4}(\log y)^{-2}}{D^{5/24}x^{1/12}}+2.
	\end{align*}
	Note that $\log z \asymp \log y$. Then, from \eqref{23_27} we see that
	\begin{align*} 
	\dfrac{w}{\delta_f}\big(\pi_f(x)-\pi_f(x-y)\big)\leq \dfrac{2\pi\, y}{\sqrt{D}\displaystyle\sum_{\ell<z}g(\ell)} + O\bigg(\dfrac{y}{\sqrt{D}(\log y)^3}\bigg).
	\end{align*}
	Using the class number formula \eqref{8.4} and the well-known estimate $L(1,\chi)\ll\log D\ll \log y$, we get
	\begin{align} \label{15_44} 
	\pi_f(x)-\pi_f(x-y)\leq \dfrac{\delta_f\,y}{h(-D)(L(1,\chi))^{-1}\displaystyle\sum_{\ell<z}g(\ell)} + O\bigg(\dfrac{\delta_f\,y}{h(-D)(\log y)^2}\bigg).
	\end{align}
	On the other hand, since $z\geq 1$ and $z\gg D^{1/4+\varepsilon/4}$, from \eqref{2_582} it follows that
	\begin{align*}  
	(L(1,\chi))^{-1}\displaystyle\sum_{\ell<z}g(\ell) & \geq \log z -\bigg(\dfrac{1}{8}+\dfrac{\varepsilon}{4}\bigg)\log D + O\big(1+(\log y)^{2}z^{-\varepsilon^2/32}\big) \\
	& \geq \dfrac{1}{4}\log y - \dfrac{1}{12}\log x -\bigg(\dfrac{1}{3}+\dfrac{\varepsilon}{4}\bigg)\log D + \dfrac{1}{4}\log a + O(\log\log y) \\
	& = \dfrac{1-\theta_1}{4}\log y + O(\log\log y),
	\end{align*}
	where $\theta_1$ is defined as
	$$
	\theta_1 = \dfrac{\log x}{3\log y} + \bigg(\dfrac{4}{3}+\varepsilon\bigg)\dfrac{\log D}{\log y} -  \dfrac{\log a}{\log y}.
	$$
	One can see that $9\varepsilon/4<1-\theta_1<1/4$. Therefore,
	$$
	\dfrac{1}{(L(1,\chi))^{-1}\displaystyle\sum_{\ell<z}g(\ell)}\leq \dfrac{4}{(1-\theta_1)\log y}\bigg(1+O\bigg(\dfrac{\log\log y}{\log y}\bigg)\bigg).
	$$
	Inserting this in \eqref{15_44}, we obtain the desired result.  
	
	\subsubsection{The second range \eqref{15_556}} We recall that, in the range
	\begin{align*} 
	x^{4/9}\leq y\leq x^{3/5}, 
	\end{align*}
	we are assuming that $x\geq D^{18}$. Now, we choose 
	\begin{align*} 
	z=\dfrac{a^{1/4}\, y^{1/7}(\log y)^{-2}}{D^{5/24}x^{1/28}}+2.
	\end{align*} 
	Note that $z\geq 1$, $z\gg D^{1/4+1/28}$ and $\log z\asymp \log y$. We proceed as in the previous case to obtain \eqref{15_44}. Using \eqref{2_582}, it follows that
	\begin{align*}  
	(L(1,\chi))^{-1}\displaystyle\sum_{\ell<z}g(\ell) & \geq \dfrac{1}{7}\log y - \dfrac{1}{28}\log x -\dfrac{31}{84}\log D + \dfrac{1}{4}\log a + O(\log\log y) \\
	& = \dfrac{1-\theta_2}{7}\log y + O(\log\log y),
	\end{align*}
	where $\theta_2$ is defined by
	$$
	\theta_2 = \dfrac{\log x}{4\log y} +\dfrac{31}{12}\dfrac{\log D}{\log y} - \dfrac{7}{4}\dfrac{\log a}{\log y}.
	$$
	Then, $11/96 \leq 1-\theta_2<7/12$, and we obtain
	$$
	\dfrac{1}{(L(1,\chi))^{-1}\displaystyle\sum_{\ell<z}g(\ell)}\leq \dfrac{7}{(1-\theta_2)\log y}\bigg(1+O\bigg(\dfrac{\log\log y}{\log y}\bigg)\bigg).
	$$
	Inserting this in \eqref{15_44}, we obtain the desired result.
	
	\subsection{The case $L(1,\chi)< (\log y)^{-2}$} Applying Theorem \ref{16_43_02_08} with $\ell=1$, we have that 
	$$
	\dfrac{w}{\delta_f}\big(\pi_f(x)-\pi_f(x-y)\big)\leq \displaystyle\sum_{x-y<n\leq x}r_f(n) = \dfrac{2\pi\,y}{\sqrt{D}} + O\bigg(\dfrac{x^{1/3}}{D^{1/6}}+\dfrac{D^{3/4}x^{1/4}}{a^{7/4}}\bigg).
	$$
	Then, using the class number formula \eqref{8.4} and the bound $L(1,\chi)<(\log y)^{-2}$, it follows that, in both ranges, 
	$$
	\pi_f(x)-\pi_f(x-y) \leq \bigg\{1+O\bigg(\dfrac{D^{1/3}x^{1/3}}{y}
	\bigg)+O\bigg(\dfrac{D^{5/4}x^{1/4}}{a^{7/4}y}
	\bigg)\bigg\}\dfrac{\delta_f \,y}{h(-D)}\,L(1,\chi)\ll \dfrac{y}{h(-D)(\log y)^2}.
	$$
	This implies our desired result in this case, and we conclude the proof of Theorem \ref{brun}. \qed
	
	\bigskip
	
\section{Hecke characters and Hecke $L$-functions}\label{sec:hecke}
	
	In this section, we will review the necessary background on Hecke $L$-functions, and their relation to quadratic forms, to prove Theorem \ref{thm:cramer}.
	
	\subsection{From quadratic forms to ideals of quadratic fields}
	It is well-known that there is a bijection between equivalence classes of positive definite quadratic forms, and equivalence classes of certain ideals in imaginary quadratic fields (see \cite[Section 7]{Cox} and \cite{Zag} for expositions). More precisely, let $f$ be a positive definite primitive form of discriminant $-D$, and let $K=\Q(\sqrt{-D})$ be the associated imaginary quadratic field. We can write 
	\begin{equation}\label{eq:conductorOfOrder}
	D=q^2D_K,
	\end{equation}
	where $q$ is some positive integer and $D_K$ is the absolute discriminant of $K$ over $\Q$. To describe the classes of ideals that correspond to quadratic forms, we must first introduce some notation.\footnote{\,\,\, This notation and some of our subsequent results in this section, could be given for arbitrary algebraic number fields, as in \cite{ZamHecke}. However, for simplicity, we will only state the definitions and results in the case of imaginary quadratic fields, which is the case relevant to positive definite quadratic forms.} We follow Zaman's notation in \cite{ZamHecke} and \cite{ZamT}. 
	  
 Denote by N be the norm in $K$ over $\Q$, and let $\qq $ be an integral ideal of $K$. Let $I(\qq)$ be the group of fractional ideals of $K$ relatively prime to $\qq$, and let $P_\qq$ be the group of principal ideals $(\alpha)$ of $K$ such that $\alpha$ is positive and $\alpha \equiv 1 \,( \mathrm{mod} \, \qq)$. Let
	\begin{equation}\label{eq:rayClassGroup}
	Cl(\qq):=I(\qq)/P_\qq
	\end{equation}
	be the narrow ray class group of $K$ modulo $\qq$. Additionally, let $H$ be a subgroup of $I(\qq)$ such that 
	\begin{equation}\label{eq:congruenceCG}
	P_\qq\subset H \subset I(\qq).
	\end{equation}
	For such an $H$, we call the quotient $I(\qq) /H$ a congruence class group, and we denote by $h_H:=|I(\qq) /H|$ its cardinality. Note that $I(\qq) /H\subset Cl(\qq).$ In our setting for quadratic forms, we will mainly need the above with the principal ideal $\qq=(q)$, where $q$ is given in \eqref{eq:conductorOfOrder};
	and with $H_0$ the group of principal ideals $(\alpha)$ of $K$ such that $\alpha \equiv a \,( \mathrm{mod} \, \qq)$, for some $a\in \Z$ with $((a),\qq)=1$ (that is, with $(a)$ and $\qq$ coprime). Note that $P_\qq\subset H_0 \subset I(\qq)$. With this notation, we can state the equivalence between ideals and forms.
	\begin{lemma}\label{lem:idealsForms}
		For each equivalence class of primitive positive definite quadratic forms $[f]$, there is a unique $A=A_f\in I(\qq) /H_0$ such that, for any integer $m$, $m$ is represented by $f$ if and only if there is an integral ideal $\mathfrak{a}\in A$, with $\rm N\mathfrak{a}=m$. This correspondence is bijective.
	\end{lemma}
	\begin{proof} 
		This follows from Theorem 7.7 and Proposition 7.22 of \cite{Cox}. See also \cite[pp. 144-145]{Cox} for the slightly more general framework of congruence class groups that we use here.
	\end{proof}
	
	\noindent In particular, note that $h(-D)=h_{H_0}=|I(\qq) /H_0|$, where $h(-D)$ is the number of proper equivalence classes of primitive quadratic forms of discriminant $-D.$
	
	\subsection{Hecke characters}
	
	We define a Hecke character $\chi  \,( \mathrm{mod} \, \qq)$ to be a character of the group $Cl(\qq)$, which we defined in \eqref{eq:rayClassGroup}.  Additionally, a character $\chi  \,( \mathrm{mod} \, H)$ is a character of a congruence class group $I(\qq) /H$. Given a Hecke character $\chi  \,( \mathrm{mod} \,\qq)$, abusing notation, we can extend the definition of $\chi$ to a multiplicative function over all integral ideals of $K$, such that $\chi(\nn)=0$ when $(\nn,\qq)\neq 1$, and $\chi(\nn)=1$ when $\nn\in P_\qq.$ With this correspondence, the characters $\chi  \,( \mathrm{mod} \, H)$ of a congruence class group correspond exactly to the Hecke characters mod $\qq$ such that $\chi(\mathfrak{h})=1$, for all $\mathfrak{h}\in H$. From now on, we will work with this extended definition of Hecke characters, as functions over all integral ideals.
	
	We denote the trivial character mod $\qq$ by $\chi_0$, so that $\chi_0(\nn)=1$ when $(\nn,\qq)=1$, and 0 otherwise. Given a character $\chi  \,( \mathrm{mod} \, \qq)$, there is a unique $\mathfrak{f}_\chi\,|\, \qq$, the conductor of $\chi$, such that $\chi$ is induced by a primitive character  $\chi^*\,( \mathrm{mod} \, \mathfrak{f}_\chi)$. This implies that $\chi(\nn)=\chi^*(\nn)\chi_0(\nn)$. See, for instance, \cite{ZamHecke} for further background on Hecke characters. For any congruence class group, we also have the orthogonality relations (see \cite[p. 44]{IK}): for all $A\in I(\qq) /H,$ 
	\begin{equation*} 
	\sum_{\chi\,( \mathrm{mod} \, H)} \chi(A) =\left\{
	\begin{array}{ll}
	h_H, & \text{if } A=H, \\
	0, & \text{if } A\neq H.
	\end{array}
	\right.
	\end{equation*}
	In particular, for an integral ideal $\mathfrak{a}$, we have that
	\begin{align} \label{eq:orthog}
	\sum_{\chi\,( \mathrm{mod} \, H)}\overline{\chi(A)}\chi(\mathfrak{a})=\left\{
	\begin{array}{ll}
	h_H, & \text{if } \mathfrak{a} \in A , \\
	0, & \text{if } \mathfrak{a} \notin A.
	\end{array}
	\right.
	\end{align}
	
	\smallskip
	
	\subsection{The family of Hecke $L$-functions} 
	Here, we describe the family of Hecke $L$-functions in the framework of \cite[Chapter 5]{IK}. Below, we adopt the notation
	$$\Gamma_\mathbb R(z):=\pi^{-z/2}\,\Gamma\left(\frac{z}{2}\right),$$ 
	where $\Gamma$ is the usual Gamma function. For a character $\chi  \,( \mathrm{mod} \, \qq)$, we define the function \begin{equation*}
	L(s,\chi):= \sum_\mathfrak{a} \frac{\chi(\mathfrak{a})}{\rm (N\mathfrak{a})^s}= \prod_\mathfrak{p} \bigg(1-\dfrac{\chi(\mathfrak{p})}{\rm (N\mathfrak{p})^{s}}\bigg)^{-1},
	\end{equation*} 
	where the sum and the product runs over all integral ideals $\mathfrak{a}$ and prime ideals $\mathfrak{p}$ of $K$, respectively, and both converge absolutely to $L(s,\chi)$ on $\{s\in\mathbb C \,;\,\text{Re}\,s>1\}$. When $\chi$ is primitive, it is known that $L(s,\chi)$ satisfies the following conditions (see \cite[p. 129]{IK} and \cite[Section 2]{ZamHecke}):

	\noindent(i) There exists a sequence $\{\lambda_\chi(n)\}_{n\ge1}$ of complex numbers ($\lambda_\chi(1) =1$), such that the series $$\sum_{n=1}^\infty\frac{\lambda_\chi(n)}{n^s}$$ converges absolutely to $L(s,\chi)$ on $\{s\in\mathbb C \,;\,\text{Re}\,s>1\}$. In fact, the sequence $\{\lambda_\chi(n)\}_{n\ge1}$ is defined by
	$$
	\lambda_\chi(n) = \sum_{\substack{\mathfrak{a} \\ \rm N\mathfrak{a} = n}}\chi(\mathfrak{a}).
	$$
	\noindent(ii) For each prime number $p$, there exist $\alpha_{1,\chi}(p)$ and $\alpha_{2,\chi}(p)$ in $\mathbb C$, such that $|\alpha_{j,\chi}(p)|\leq1$ and\footnote{\,\,\, This follows from the factorization law of primes in imaginary quadratic fields (see, for instance \cite[p. 57]{IK}).} 
	$$L(s,\chi)=\prod_p\left(1-\frac{\alpha_{1,\chi}(p)}{p^s}\right)^{-1}\left(1-\frac{\alpha_{2,\chi}(p)}{p^s}\right)^{-1}.$$
	The product converges absolutely on the half plane $\{s \in \C; \text{Re}\,s>1\}$.

	\noindent(iii) Denote $D_\chi:= D_K\,\rm N\cond$. The completed $L$-function 
	$$\Lambda(s,\chi):=D_\chi^{s/2}\,\Gamma_\mathbb R(s)\,\Gamma_\mathbb R(s+1)\,L(s,\chi)$$
	is a meromorphic function of order 1. It has no poles other than $0$ and $1$, which have the same order $r(\chi)\in\{0,1\}$. Additionally, $r(\chi)=1$ if $\chi$ is the trivial character mod $\qq$, and 0 otherwise. Furthermore, the function $\Lambda(s,\chi)$ satisfies the  the functional equation
	\begin{equation*} 
	\Lambda(s,\chi)=\epsilon(\chi)\Lambda(1-s,\overline{\chi}),
	\end{equation*}
	where $\epsilon(\chi)$ is a complex number of absolute value 1. In particular, when $\qq = (1)$ and $\chi=\chi_0$, the function $L(s,\chi_0)$ is the Dedekind zeta function $\zeta_K(s)$ of $K$, defined as in \cite[Section 5.10]{IK}. Moreover, we have that
	\begin{equation*} 
	\frac{L'}{L}(s,\chi)=-\sum_{n=2}^\infty\frac{\Lambda_\chi(n)}{n^s}
	\end{equation*}
	converges absolutely for $\re{s}>1$, where\footnote{\,\,\, We also extend this definition of $\Lambda_\chi$ to any function defined over integral ideals, in place of $\chi$.}\begin{align}  \label{eq:mangoldtChi}
	\Lambda_\chi(n)= \sum_{\substack{\mathfrak{a} \\ \rm N\mathfrak{a} = n}}\chi(\mathfrak{a})\Lambda_K(\mathfrak{a}), 
	\end{align}
	and
	\begin{equation}  \label{eq:mangoldtK}
	\Lambda_K(\mathfrak{a})=\left\{
	\begin{array}{ll}
	\log \rm N\mathfrak{p},\hspace{0.3cm}\text{if}\hspace{0.2cm} \mathfrak{a}=\mathfrak{p}^r \hspace{0.2cm}\mbox{for some integer $r\geq 1$},\\
	0,\hspace{1.06cm}\text{otherwise}.
	\end{array}
	\right.
	\end{equation}
	Logarithmically differentiating the Euler product, it can be shown that $|\Lambda_\chi(n)|\leq 2\Lambda(n)$, where $\Lambda(n)$ is the usual von Mangoldt function. In particular, one can see that
	\begin{align} \label{1_13_27_11}
	\sum_{\substack{\mathfrak{a} \\ \rm N\mathfrak{a} = n}}\Lambda_K(\mathfrak{a}) \leq 2\Lambda(n).
	\end{align}

	\begin{remark}
		If $\chi$ (mod $\qq$) is a non-primitive character induced by the primitive character $\chi^*$ (mod $\cond$), we have the relation
		\begin{equation} \label{20_16}
		L(s,\chi) = L(s,\chi^*)\prod_{\pp| \qq}\left(1-\frac{\chi^*(\pp)}{(\rm N \pp)^{s}}
		\right).
		\end{equation}
		In particular, $L(s,\chi)$ also extends to a meromorphic function, such that $L(s,\chi)$ and $L(s,\chi^*)$ have the same set of zeros in the strip $0 < \re s <1. $
	\end{remark}
	\subsection{The Guinand-Weil formula}
	The classical Guinand-Weil explicit formula establishes a relationship between the zeros of an $L$-function, the associated coefficients $\Lambda_\chi(n)$ (given in this case in \eqref{eq:mangoldtChi}), an arbitrary function $G$, and its Fourier transform $\widehat{G}$. 
	In the case of a Hecke $L$-function $L(s,\chi)$, the coefficients $\Lambda_\chi(n)$ contain information about prime ideals, twisted by the character $\chi$. 
	
 	We will use the version of this formula in \cite[Lemma 5]{CF}. 
	However, this only applies to the case of a primitive Hecke character mod $\qq$, and we will need a version that averages over all characters, primitive and non-primitive, in a given congruence class group. The result is the following, which could be of independent interest for further applications.\footnote{\,\,\, Like the rest of this section, the previous lemma is only stated for the case of imaginary quadratic fields, to simplify the technical details of some of the definitions. However, a similar statement holds true for families of Hecke $L$-functions of arbitrary algebraic number fields, with a similar proof.}
\begin{lemma}\label{lem:GWaverage} Let $\qq$ be an integral ideal of the imaginary quadratic field $K$. Let $I(\qq)/H$ be a congruence class group as in \eqref{eq:congruenceCG}, and let $A\in I(\qq) /H.$ Let $G(s)$ be analytic in the strip $|\im{s}|\leq \tfrac12+\varepsilon$, for some $\varepsilon>0$. Assume that $|G(s)|\ll(1+|s|)^{-(1+\delta)}$ for some $\delta>0$, when $|\re{s}|\to\infty$. Then 
		\begin{align*}
		\sum_{\chi\,( \mathrm{mod} \, H)} \overline{\chi(A)} \sum_{\rho_\chi} G\left(\frac{\rho_\chi - \tfrac12}{i}\right)= \,\,& G\left(\frac{1}{2i}\right)+G\left(-\frac{1}{2i}\right) \\  &+ 
		\frac{h_H\kappa_H(A)}{\pi}
		\int_{-\infty}^\infty G(u)\bigg\{\,{\rm Re}\,\frac{\Gamma_\mathbb R'}{\Gamma_\mathbb R}\left(\hh+iu\right)+\,{\rm Re}\,\frac{\Gamma_\mathbb R'}{\Gamma_\mathbb R}\left(\tfrac{3}{2}+iu\right)\bigg\}\d u\\
		&  -\frac{h_H}{2\pi}\sum_{n=2}^\infty\frac{1}{\sqrt{n}}\, \widehat G\left(\frac{\log n}{2\pi}\right)\left\{\displaystyle\sum_{\substack{\mathfrak{a}\in A  \vspace{0.07cm}\\ \rm N\mathfrak{a}=n }}\Lambda_K(a) \right \}\\
		&   -\frac{1}{2\pi}\sum_{n=2}^\infty\frac{1}{\sqrt{n}}\,\widehat G\left(\frac{-\log n}{2\pi}\right)\left\{\sum_{\chi\,( \mathrm{mod} \, H)} \overline{\chi(A)}\ov{\Lambda_{\chi}(n)}\right\} \\
		& +O\left(h_H\log (D_K\rm N \qq)\lVert \widehat{G}\rVert_\infty\right),
		\end{align*}
		where the sum over $\rho_\chi$ runs over all zeros of $L(s,\chi)$ in the strip $0<\re{s}<1$. The coefficients $\Lambda_\chi(n)$ and $\Lambda_K(\mathfrak{a})$ are defined in \eqref{eq:mangoldtChi}; $\kappa_H(A)=1$ when $A=H$, and $0$ otherwise.
	\end{lemma}
	
	
	\begin{proof}
		We follow the approach used in \cite[Lemma 3]{CPL} for Dirichlet characters modulo $q\geq 3$. The Guinand-Weil formula in \cite[Lemma 5]{CF}, when specialized to $L(s,\chi)$ for a primitive Hecke character $\chi\,( \mathrm{mod} \, \qq)$, states the following:
		\begin{align}\label{eq:guinand}
		\begin{split}
		\sum_{\rho_\chi} G\left(\frac{\rho_\chi - \tfrac12}{i}\right)&= r(\chi)\left\{G\left(\frac{1}{2i}\right)+G\left(-\frac{1}{2i}\right)\right\} + \frac{\log D_\chi}{2\pi}\,\widehat{G}(0) \\
		& \ \ \  + \frac{1}{\pi}\int_{-\infty}^\infty G(u)\bigg\{\,{\rm Re}\,\frac{\Gamma_\mathbb R'}{\Gamma_\mathbb R}\left(\hh+iu\right)+\,{\rm Re}\,\frac{\Gamma_\mathbb R'}{\Gamma_\mathbb R}\left(\tfrac{3}{2}+iu\right)\bigg\}\d u\\
		& \ \ \ -\frac{1}{2\pi}\sum_{n=2}^\infty\frac{1}{\sqrt{n}}\left\{\Lambda_{\chi}(n)\, \widehat G\left(\frac{\log n}{2\pi}\right)+\overline{\Lambda_{\chi}(n)}\, \widehat G\left(\frac{-\log n}{2\pi}\right)\right\},
		\end{split}
		\end{align}
where the sum on the left-hand side runs over the zeros of $\Lambda(s,\chi)$, which coincide with the zeros of $L(s,\chi)$ in $0<\re{s}<1$. Now, let $\chi$ be a non-primitive character mod $\qq$. Let $\chi^* \,(\mathrm{mod} \,\cond)$ be the unique primitive character that induces $\chi,$ where $\cond \,|\, \qq,$ so that $\chi = \chi^*\chi_0,$ where $\chi_0$ is the trivial character mod $\qq$. We can then write $\chi^*(\mathfrak{a})=\chi(\mathfrak{a})+\chi^*(\mathfrak{a})\Tilde{\chi}_0(\mathfrak{a})$, where $\Tilde{\chi}_0(\mathfrak{a})=1-\chi_0(\mathfrak{a})$. 
		Applying \eqref{eq:guinand} for $\chi^*$, it follows that
		\begin{align*}
	\sum_{\rho_{\chi^*}} G\left(\frac{\rho_{\chi^*}- \tfrac12}{i}\right)&=  r(\chi)\left\{G\left(\frac{1}{2i}\right)+G\left(-\frac{1}{2i}\right)\right\}+
		\frac{\log D_{\chi^*}}{2\pi}\,\widehat{G}(0)  \\
		& \ \ \ + \frac{1}{\pi}\int_{-\infty}^\infty G(u)\bigg\{\,{\rm Re}\,\frac{\Gamma_\mathbb R'}{\Gamma_\mathbb R}\left(\hh+iu\right)+\,{\rm Re}\,\frac{\Gamma_\mathbb R'}{\Gamma_\mathbb R}\left(\tfrac{3}{2}+iu\right)\bigg\}\d u\\
		& \ \ \ -\frac{1}{2\pi}\sum_{n=2}^\infty\frac{1}{\sqrt{n}}\left\{\Lambda_{\chi}(n)\, \widehat G\left(\frac{\log n}{2\pi}\right)+\overline{\Lambda_{\chi}(n)}\, \widehat G\left(\frac{-\log n}{2\pi}\right)\right\} \\	
		& \ \ \ -\frac{1}{2\pi}\sum_{n=2}^\infty\frac{1}{\sqrt{n}}\left\{\Lambda_{\chi^*\Tilde{\chi}_0}(n)\, \widehat G\left(\frac{\log n}{2\pi}\right)+\overline{\Lambda_{\chi^*\Tilde{\chi}_0}(n)}\, \widehat G\left(\frac{-\log n}{2\pi}\right)\right\} .
		\end{align*}
		Since $\Tilde{\chi}_0(\mathfrak{a})=0$ when $\mathfrak{a}$ and $\qq$ are coprime, by Lemma \ref{lemma10}, the last sum can be bounded by
		\begin{align*}
		\left| \frac{1}{2\pi}\sum_{n=2}^\infty\frac{1}{\sqrt{n}}\left\{\Lambda_{\chi^*\Tilde{\chi}_0}(n)\, \widehat G\left(\frac{\log n}{2\pi}\right)+\overline{\Lambda_{\chi^*\Tilde{\chi}_0}(n)}\, \widehat G\left(\frac{-\log n}{2\pi}\right)\right\} \right|
		&\ll \lVert \widehat{G}\rVert_\infty
		\sum_{\pp | \qq, \, k\ge 1 } \frac{\log \rm N\pp}{(\rm N\pp)^{k/2}} \\
		&\ll 
		\lVert \widehat{G}\rVert_\infty \,
		\sqrt{\log (\rm N \qq+1)}.
		\end{align*}
		Letting $Q_H:=\max \{\rm N\mathfrak{f}_\chi:\, \chi\,( \mathrm{mod} \, \it{H}) \}$, note that 
		$$\log D_\chi \le \log (D_K\,Q_H) \le \log( D_K \,\rm N \qq).$$ 
		By \eqref{20_16}, $L(s,\chi)$ and $L(s,\chi^*)$ have the same zeros in $0<\re{s}<1$. Then, for any non-primitive character $\chi\, (\mathrm{mod}\, H)$, we obtain that\footnote{\,\,\, In fact, in this step we have the slightly better error term $\ll \big(\log(D_K \, Q_H) + \sqrt{\log \rm (N \qq+1)} \big)\|\widehat{G}\|_\infty$.} 
		\begin{align*}
		\sum_{\rho_{\chi}} G\left(\frac{\rho_{\chi}- \tfrac12}{i}\right)&=   r(\chi)\left\{G\left(\frac{1}{2i}\right)+G\left(-\frac{1}{2i}\right)\right\} +  \frac{1}{\pi}\int_{-\infty}^\infty G(u)\bigg\{\,{\rm Re}\,\frac{\Gamma_\mathbb R'}{\Gamma_\mathbb R}\left(\hh+iu\right)+\,{\rm Re}\,\frac{\Gamma_\mathbb R'}{\Gamma_\mathbb R}\left(\tfrac{3}{2}+iu\right)\bigg\}\d u \\
		& \ \ \ -\frac{1}{2\pi}\sum_{n=2}^\infty\frac{1}{\sqrt{n}}\left\{\Lambda_{\chi}(n)\, \widehat G\left(\frac{\log n}{2\pi}\right)+\overline{\Lambda_{\chi}(n)}\, \widehat G\left(\frac{-\log n}{2\pi}\right)\right\} + O\left(\log (D_K\, \rm N\qq)\lVert \widehat{G}\rVert_\infty\right).
		\end{align*}
		We now multiply by $\chi(A)$ and sum over all $\chi\,(\mathrm{mod} \,H)$. Using that $r(\chi)=1$ if $\chi$ is the trivial character, and 0 otherwise, we get that 
		\begin{align*}
		\sum_{\chi\,( \mathrm{mod} \, H)} \overline{\chi(A)} \sum_{\rho_\chi} G\left(\frac{\rho_\chi - \tfrac12}{i}\right) = \,\, &  G\left(\frac{1}{2i}\right)+G\left(-\frac{1}{2i}\right)   
		\\ &+ 
		\sum_{\chi\,( \mathrm{mod} \, H)} \overline{\chi(A)}\frac{1}{\pi}
		\int_{-\infty}^\infty G(u)\bigg\{\,{\rm Re}\,\frac{\Gamma_\mathbb R'}{\Gamma_\mathbb R}\left(\hh+iu\right)+\,{\rm Re}\,\frac{\Gamma_\mathbb R'}{\Gamma_\mathbb R}\left(\tfrac{3}{2}+iu\right)\bigg\}\d u\\
		& -\frac{1}{2\pi}\sum_{n=2}^\infty\frac{1}{\sqrt{n}}\, \widehat G\left(\frac{\log n}{2\pi}\right)	\left\{\sum_{\chi\,( \mathrm{mod} \, H)}\overline{\chi(A)}\Lambda_\chi(n)\right\}\\
		& -\frac{1}{2\pi}\sum_{n=2}^\infty\frac{1}{\sqrt{n}}\,\widehat G\left(\frac{-\log n}{2\pi}\right)\left\{\sum_{\chi\,( \mathrm{mod} \, H)} \overline{\chi(A)}\ov{\Lambda_{\chi}(n)}\right\}	\\
		&  +O\left(h_H\log (D_K \rm N\qq)\,  \lVert \widehat{G}\rVert_\infty\right).
		\end{align*}
		Using \eqref{eq:mangoldtChi}, Fubini's theorem, and the orthogonality relations \eqref{eq:orthog}, we obtain the desired result.
	\end{proof}

	\section{Proof of Theorem \ref{thm:cramer}}\label{sec:cramer}
	We follow the argument of Carneiro, Milinovich, and Soundararajan in \cite[Section 5]{CMS}. To begin, fix a primitive positive definite quadratic form $f$ of discriminant $-D$, and let $A\in I(\qq)/H_0$ be the corresponding ideal class as in Lemma \ref{lem:idealsForms}.
	Assume GRH for all Hecke $L$-functions associated with characters $\chi \,( \mathrm{mod} \, H_0)$. Furthermore, take a fixed even and bandlimited Schwartz function $F:\R \to\R$ such that $F(0)>0$ and $\supp(\widehat{F})\subset [-N,N]$, for some $N\geq 1$. We can extend $F$ to an entire function, and using the Phragm\'{e}n-Lindel\"{o}f principle, the hypotheses of Lemma \ref{lem:GWaverage}  are satisfied. Let $0<\Delta\leq 1$ and $1<\sigma$ be free parameters, to be chosen later, such that
	\begin{equation*} 
	2\pi\Delta N\leq \log \sigma.
	\end{equation*}
	We remark that we will send $\sigma\to\infty$ and $\Delta\to0$. In this section, we will allow all implicit constants to depend on the fixed quadratic form $f$, its discriminant, and the fixed function $F$, but not on the free parameters $\sigma$ and $\Delta.$ 
	
\subsection{Asymptotic analysis}	The following computations are similar to those in \cite{CMS} and \cite{CPL}, so we highlight the differences. Consider the function $G(z):=\Delta F(\Delta z)\sigma^{iz}$.  We apply Lemma \ref{lem:GWaverage} to $G$, with our specific choices of $\qq$, $H_0$ and $A$. This gives  
	\begin{align}\label{eq:guinGRH}
	\begin{split}
	\sum_{\chi\,( \mathrm{mod} \, H_0)} \overline{\chi(A)} \sum_{\gamma_\chi} G(\gamma_\chi)=\,\,& G\left(\frac{1}{2i}\right)+G\left(-\frac{1}{2i}\right)   \\
	&+ 
	\frac{h(-D)\kappa_{H_0}(A)}{\pi}
	\int_{-\infty}^\infty G(u)\bigg\{\,{\rm Re}\,\frac{\Gamma_\mathbb R'}{\Gamma_\mathbb R}\left(\hh+iu\right)+\,{\rm Re}\,\frac{\Gamma_\mathbb R'}{\Gamma_\mathbb R}\left(\tfrac{3}{2}+iu\right)\bigg\}\d u\\
	&  -\frac{h(-D)}{2\pi}\sum_{n=2}^\infty\frac{1}{\sqrt{n}}\, \widehat G\left(\frac{\log n}{2\pi}\right)\left\{\sum_{\substack{\mathfrak{a}\in A \\ \rm {N}\mathfrak{a} = n}}\Lambda_K(\mathfrak{a})\right\}\\
	&  -\frac{1}{2\pi}\sum_{n=2}^\infty\frac{1}{\sqrt{n}}\widehat G\left(\frac{-\log n}{2\pi}\right)\left\{\sum_{\chi\,( \mathrm{mod} \, H_0)} \overline{\chi(A)}\ov{\Lambda_{\chi}(n)}\right\}	+O(1),
	\end{split}
	\end{align}
	where the inner sum on the left-hand side runs over the imaginary parts of the zeros of $L(s,\chi)$ on the line $\re{s}=\hh$. The first, second, and fourth lines in the right-hand side of \eqref{eq:guinGRH} can be estimated as in \cite[pp. 553--554]{CMS}. In this way, we obtain the following:
	%
	%
	%
	%
	%
	%
	\begin{align*} 
	\sum_{\chi\,( \mathrm{mod} \, H_0)}\overline{\chi(A)}\sum_{\gamma_{\chi}} G(\gamma_{\chi}) 
	=\,\, &\Delta  F(0)\, \sqrt{\sigma}
	\ -\ \frac{h(-D)}{2\pi}\sum_{n=2}^\infty\frac{1}{\sqrt{n}}\,\widehat G\left(\frac{\log n}{2\pi}\right)
	\left\{\displaystyle\sum_{\substack{\mathfrak{a}\in A  \vspace{0.07cm}\\ \rm N\mathfrak{a}=n }}\Lambda_K(\mathfrak{a}) \right \} \\
	&+ O\big(\Delta^2\sqrt{\sigma}\big)+ O(1).
	\end{align*}
	Therefore,
	\begin{align} \label{eq:UsingGW}
	\begin{split}
	\Delta F(0)\,\sqrt{\sigma}\leq & \displaystyle\sum_{\chi\,( \mathrm{mod} \, H_0)}\displaystyle\sum_{\gamma_\chi} \big|G(\gamma_\chi)\big| + \frac{h(-D)}{2\pi}\sum_{n=2}^\infty\frac{1}{\sqrt{n}}\,\widehat G\left(\frac{\log n}{2\pi}\right)_{+}\left\{\displaystyle\sum_{\substack{\mathfrak{a}\in A  \vspace{0.07cm}\\ \rm N\mathfrak{\mathfrak{a}}=n }}\Lambda_K(\mathfrak{a})\right\}\\
	& \,\,\,\,\,\,\,\,
	+O\big(\Delta^2\sqrt{\sigma}\big)+ O(1).
	\end{split}
	\end{align}
	To analyze the first sum on the right-hand side of \eqref{eq:UsingGW}, we recall the formula \cite[Theorem 5.8]{IK}
	$$
	N(T,\chi)=\dfrac{T}{\pi}\log\bigg(\dfrac{D_\chi T^2}{(2\pi e)^2}\bigg) + O(\log T+\log D_\chi),
	$$
	where $N(T,\chi)$ denotes the number of zeros of $L(s,\chi)$ in the rectangle $0< \sigma< 1$ and $|\gamma|\leq T$. This holds for both primitive and non-primitive characters. Note that the term $T^2$ comes from the fact that $K$ is an algebraic extension of $\Q$ of degree 2. 
	For each $\chi \,( \mathrm{mod} \, H_0)$, integration by parts gives us (see \cite[Eq. (5.4)]{CMS}) that
	\begin{align*} \displaystyle\sum_{\gamma_\chi} |G(\gamma_\chi)| & = \frac{\log(1/2\pi\Delta)}{\pi}\norm{F}_1 +O(1).
	\end{align*}
	Then,
	\begin{align} \label{eq:sumzeros}
	\displaystyle\sum_{\chi\,( \mathrm{mod} \, H_0)}\displaystyle\sum_{\gamma_\chi} \big|G(\gamma_\chi)\big| = h(-D)\,\frac{\log(1/2\pi\Delta)}{\pi}\norm{F}_1 +O(1).
	\end{align}
\subsection{From ideals to primes represented by $f$}	We now consider the second sum on the right-hand side of \eqref{eq:UsingGW}. This sum is given by
	\begin{align} \label{eq:sumPrimes}
	\sum_{n=2}^\infty\frac{1}{\sqrt{n}}\,\widehat G\left(\frac{\log n}{2\pi}\right)_{+}\left\{\displaystyle\sum_{\substack{\mathfrak{a}\in A  \vspace{0.07cm}\\ \rm N\mathfrak{a}=n }}\Lambda_K(\mathfrak{a})\right\}
	=
	\sum_{n=2}^\infty\frac{1}{\sqrt{n}}\,\widehat F\left(\frac{\log(n/\sigma)}{2\pi\Delta}\right)_{+}\left\{\displaystyle\sum_{\substack{\mathfrak{a}\in A  \vspace{0.07cm}\\ \rm N\mathfrak{a}=n }}\Lambda_K(\mathfrak{a})\right\}.
	\end{align}
	
	We first make some reductions to the sum over $n$. Initially, since $\supp(\widehat{F})\subset [-N,N]$, the sum runs over $\sigma e^{-2\pi\Delta N}\leq n\leq \sigma e^{2\pi\Delta N}$. Note that the sum is supported over integers $n$ that are (integer) prime powers, since $\Lambda_K$ is supported on powers of prime ideals. Furthermore, by the relationship between ideals and forms (Lemma \ref{lem:idealsForms}), the sum over $n$ is actually supported over prime powers that are represented by $f$. Using \eqref{1_13_27_11}, the contribution of the prime powers $n=p^k$, with $k\ge 2$, is $O(1)$. The sum \eqref{eq:sumPrimes} is therefore reduced, up to an error term $O(1)$, to a sum over primes $p$ represented by $f$, such that $p \in [\sigma e^{-2\pi\Delta N}, \sigma e^{2\pi\Delta N}]$. Our version of the Brun-Titchmarsh theorem, Corollary \ref{cor:brun}, will be useful to estimate the contribution near the endpoints of this interval.
	
	We continue by choosing the parameters $\Delta$ and $\sigma,$ and bounding the corresponding contribution of the primes in the interval $(\sigma e^{-2\pi\Delta},\sigma e^{2\pi\Delta}]$ 
	to the sum \eqref{eq:sumPrimes}. Fix $\alpha\geq 0$, and assume that $c>0$ is a fixed constant such that
	$$
	\liminf_{x\to\infty}\dfrac{\pi_f\big(x+c\sqrt{x}\log x)-\pi_f(x)}{\sqrt{x}} \le  \alpha.
	$$
	Then, for any $\varepsilon>0$, there exists a sequence of $x\to\infty$, such that there are at most $(\alpha + \varepsilon)\sqrt{x}$ primes represented by $f$ in the interval $(x, x+c\sqrt{x}\log x].$ For each $x$ in this sequence, we choose $\sigma$ and $\Delta$ such that
	\begin{align*}
	\left[x,x+c\sqrt{x}\log x\right]=\Big[\sigma e^{-2\pi\Delta},\sigma e^{2\pi\Delta}\Big].
	\end{align*}	
	This implies that (see \cite[Eq. (5.7)-(5.8)]{CMS})
	\begin{align*}
	4\pi\Delta=c\,\dfrac{\log x}{\sqrt{x}} + O\bigg(\dfrac{\log^2x}{x}\bigg),\   \ \text{ and }   \ \ \sigma=x+O(\sqrt{x}\log x). 
	\end{align*}
	Note that $\delta_f=1/2$ (defined in \eqref{20_39}) if and only if $A=\{\bar{\mathfrak{a}}: \mathfrak{a}\in A\}$. Using the factorization law of primes in imaginary quadratic fields \cite[p. 57]{IK}, we can plainly see that
	\begin{equation}\label{eq:mangoldt}
	\displaystyle\sum_{\substack{\mathfrak{a}\in A  \vspace{0.07cm}\\ \rm N\mathfrak{a}=p }}\Lambda_K(\mathfrak{a})
	= 
	\log p  \displaystyle\sum_{\substack{\mathfrak{a}\in A  \vspace{0.07cm}\\ \rm N\mathfrak{a}=p }} 1 
	\le \frac{\log p}{\delta_f}.
	\end{equation}
	Using that $(\widehat{F}(t))_+\leq\|F\|_1$ and \eqref{eq:mangoldt}, we bound the contribution in this interval by
	\begin{align*}
	\|F\|_1\displaystyle\sum_{p\in(\sigma e^{-2\pi\Delta}, \sigma e^{2\pi\Delta}]}\dfrac{1}{\sqrt{p}}\left\{\displaystyle\sum_{\substack{\mathfrak{a}\in A  \vspace{0.07cm}\\ \rm N\mathfrak{a}=p }}\Lambda_K(\mathfrak{a})\right\}& \leq \frac{\|F\|_1}{\delta_f}\,\displaystyle\sum_{p\in(\sigma e^{-2\pi\Delta},\sigma e^{2\pi\Delta}]}\dfrac{\log p}{\sqrt{p}} \\& \leq \frac{\|F\|_1}{\delta_f}(\alpha+\varepsilon)\sqrt{x}\,\,\dfrac{\log x}{\sqrt{x}}=\frac{\|F\|_1}{\delta_f}(\alpha+\varepsilon)\log x.
	\end{align*}
	Finally, we estimate the contribution of the primes in the intervals $[\sigma e^{-2\pi\Delta N},\sigma e^{-2\pi\Delta}]$ and $[\sigma e^{2\pi\Delta},\sigma e^{2\pi\Delta N}]$. We will need the following estimate: for $g\in C^1([a,b])$ we have
	\begin{equation} \label{eq:sumIntegral}
	0\leq S(g_+,P)-\displaystyle\int_a^b(g(t))_+\,dt\leq \delta(b-a)\sup_{x\in [a,b]}|g^{\prime}(x)|,
	\end{equation}
	where $P$ is a partition of $[a,b]$ of norm at most $\delta$ and $S(g_+,P)$ is the upper Riemann sum of the function $g_+$ and the partition $P$. We apply \eqref{eq:sumIntegral} with the function
	$$g(t)=\widehat{F}\bigg(\frac{\log (t/\sigma)}{2\pi\Delta}\bigg),$$
	and the partition $P=\{x_0<\ldots<x_J\}$ that covers the interval $[\sigma e^{2\pi\Delta},\sigma e^{2\pi\Delta N}]\subset \cup_{j=0}^{J-1}[x_j,x_{j+1}]$, with $x_0=\sigma e^{2\pi\Delta}$, $x_{j+1}=x_j+\sqrt{x_j}$. Defining $M_j=\sup\{g^+(x) : x\in[x_j,x_{j+1}]\}$, by Corollary \ref{cor:brun}, \eqref{eq:mangoldt}, and \eqref{eq:sumIntegral} we bound the contribution in this interval as follows:\footnote{\,\,\, See  \cite[p. 7]{CPL} for details in this computation.}
	%
	\begin{align*}  
	\sum_{\substack{1\leq\frac{\log p/\sigma}{2\pi\Delta}\leq N }} & 
	\frac{1}{\sqrt{p}}
	\,\widehat F \left(\frac{\log(p/\sigma)}{2\pi\Delta}\right)_{+} \left\{\displaystyle\sum_{\substack{\mathfrak{a}\in A  \vspace{0.07cm}\\ \rm N\mathfrak{a}=p }}\Lambda_K(\mathfrak{a})\right\} \nonumber \\
	&\le \frac{1}{\delta_f}\sum_{\substack{1\leq\frac{\log p/\sigma}{2\pi\Delta}\leq N \\ p \text{ represented by }f}}\, \widehat F \left(\frac{\log(p/\sigma)}{2\pi\Delta}\right)_{+} \frac{\log p}{\sqrt{p}} \nonumber \\
	&\leq\displaystyle\sum_{j=0}^{J-1}\bigg(\dfrac{\log x_j}{\sqrt{x_j}}M_j\bigg)\dfrac{(28+\varepsilon)\sqrt{x_j}}{h(-D)\log x_j}  \leq
	\dfrac{(28+\varepsilon)\,\sqrt{\sigma}\,(2\pi\Delta)}
	{h(-D)}
	\,\displaystyle\int_1^{N}(\widehat{F}(t))_+\,\d t+O(1).
	\end{align*}
	We treat the other interval in a similar way. Combining the two intervals, we obtain
	%
	%
	\begin{align*}
	\displaystyle\sum_{\substack{1<\left|\frac{\log(p/\sigma)}{2\pi\Delta}\right|\leq N }}\dfrac{1}{\sqrt{p}}\,(\widehat{F})_+\bigg(\dfrac{\log(p/\sigma)}{2\pi\Delta}\bigg)\left\{\displaystyle\sum_{\substack{\mathfrak{a}\in A  \vspace{0.07cm}\\ \rm N\mathfrak{a}=p }}\Lambda_K(\mathfrak{a})\right\} &  
	\leq
	\dfrac{(28+\varepsilon)\,\sqrt{\sigma}\,(2\pi\Delta)}
	{h(-D)}
	\int_{[-1,1]^c}(\widehat{F}(t))_+\,\d t + O(1).
	\end{align*}
	Grouping the previous estimates, we conclude that
	\begin{align} \label{eq:sumPrimesAns}
	\begin{split}
	\sum_{n=2}^\infty\frac{1}{\sqrt{n}}\,(\widehat G)_+\left(\frac{\log n}{2\pi}\right)\left\{\displaystyle\sum_{\substack{\mathfrak{a}\in A  \vspace{0.07cm}\\ \rm N\mathfrak{a}=n }}\Lambda_K(\mathfrak{a})\right\}
	&\leq 
	\frac{\norm{F}_1}{\delta_f}(\alpha+\varepsilon)\log x \\
	&\,\,\,\,\,\,\,\,\,\,+ \dfrac{(28+\varepsilon)\,\sqrt{\sigma}\,(2\pi\Delta)}
	{h(-D)}\int_{[-1,1]^c} 
	(\widehat{F}(t))_+\,\d t+ O(1).
	\end{split}
	\end{align}
	Then, inserting the estimates \eqref{eq:sumzeros} and \eqref{eq:sumPrimesAns} in \eqref{eq:UsingGW}, and reordering the terms, we obtain
	\begin{align*}
	\sqrt{\sigma}\Delta\bigg(F(0)-(28+\varepsilon)\int_{[-1,1]^c}(\widehat{F}(t))_+\,\d t\bigg) 
	\leq & \,
	\frac{h(-D)\norm{F}_1}{2\pi}\left[
	(\alpha+\varepsilon)\frac{\log x}{\delta_f} + 2\log(1/2\pi\Delta)
	\right] +O(1).
	\end{align*}
	Sending $x\to\infty$ along the sequence, and then sending $\varepsilon\to0$, we obtain that
	\begin{equation}\label{eq:finalInequalityF}
	c\leq 2\dfrac{(\delta_f+\alpha)h(-D)}{\delta_f}\dfrac{\norm{F}_1}{F(0)-28\int_{[-1,1]^c}(\widehat{F}(t))_+ \d t},
	\end{equation}
	where we assume that the denominator is positive. By the approximation argument in \cite[Section 4.1]{CMS}, equation \eqref{eq:finalInequalityF} also holds for any even continuous function $F\in L^1(\R)$, with the mentioned restriction on the denominator. Now we must find a suitable function $F$. 
	
\subsection{Construction of $F$} Inspired by Gorbachev's constructions in \cite{Gor} for a related Fourier optimization problem (see also the remark in \cite[p. 536]{CMS}), we search numerically for optimal dilations of functions of the form 
	\begin{equation} \label{eq:H(x)}
	H(x)=\cos(2\pi x)\sum_{j=1}^n \frac{a_j}{(2j-1)^2-16x^2}.
	\end{equation}
	Using a greedy algorithm, we found the function \begin{align} \label{eq:finalInequalityF3}
	 F(x)=H\left(\frac{x}{0.98644}\right),
	\end{align}
	where 
	\[H(x)= \cos(2\pi x)\left(
	\frac{68}{1-16x^2} + \frac{5}{9-16x^2}+\frac{1}{25-16x^2}
	\right),
	\]
	which, by numerical experiment,\footnote{\,\,\, The bound 0.91833 in \eqref{eq:finalInequalityF2} was determined rigorously, using ball arithmetic with the ARB library.}  gives
	\begin{equation}\label{eq:finalInequalityF2}
	\dfrac{\norm{F}_1}{F(0)-28\int_{[-1,1]^c}(\widehat{F}(t))_+\,\d t}<0.91833.
	\end{equation}
	Therefore, inserting it in \eqref{eq:finalInequalityF} we conclude the desired result.  \qed
	
	\bigskip
	
	\section{Uncertainty and Fourier optimization}\label{sec:uncertainty}
	In this section, we discuss some qualitative aspects on the problem of choosing an optimal function $F$ in \eqref{eq:finalInequalityF2}. For $1\leq A< \infty$, in \cite{CMS} the authors introduced the functionals
	\begin{equation*} 
	J_A(F):=\frac{|F(0)| -A\int_{[-1,1]^c}|\widehat{F}(t)|\, \d t}{\|F\|_1}
	\end{equation*}
	and 
	\begin{equation*} 
	J_A^+(F):= \frac{F(0) -A\int_{[-1,1]^c}(\widehat{F}(t))_+ \,\d t}{\|F\|_1},
	\end{equation*}
	where $F$ is a continuous function such that $F\in L^1(\R)\setminus\{0\}$.  They considered the following problems:
	
	\begin{problem}
		Define $\mc{A}$ to be the class of continuous functions $F:\R\rightarrow \C$, with $F\in L^1(\R)\setminus\{0\}$, and $\mc{E}=\{F\in\mc{A}: \, \supp \widehat{F}\subset [-1,1] \}$. Find
		\begin{equation*} 
		\mc{C}(A):=\left\{
		\begin{array}{ll}
		\displaystyle\sup_{F\in \mc{A}} J_A(F),\hspace{0.3cm}\text{if}\hspace{0.2cm} 1\le A<\infty; \vspace{0.2cm}\\ 
		\displaystyle\sup_{F\in \mc{E}} \frac{|F(0)|}{\|F\|_1},\hspace{0.3cm}\text{if}\hspace{0.2cm}  A=\infty.
		\end{array}
		\right.
		\end{equation*}
	\end{problem}
	\begin{problem}
		Define $\mc{A^+}$ to be the class of even and continuous functions $F:\R\rightarrow \R$, with $F\in L^1(\R)\setminus\{0\}$, and $\mc{E^+}=\{F\in \mc{A^+}: \, \widehat{F}(t)\le 0 \  \mbox{for} \hspace{0.15cm} |t|\ge 1\}$. Find
		\begin{equation*} 
		\mc{C^+}(A):=\left\{
		\begin{array}{ll}
		\displaystyle\sup_{F\in \mc{A}} J^+_A(F),\hspace{0.3cm}\text{if}\hspace{0.2cm} 1\le A<\infty; \vspace{0.2cm}\\ 
		\displaystyle\sup_{F\in \mc{E^+}} \frac{F(0)}{\|F\|_1},\hspace{0.3cm}\text{if}\hspace{0.2cm}  A=\infty.
		\end{array}
		\right.
		\end{equation*}
	\end{problem}
	
	\smallskip

	The authors show that, for all $1\le A \le \infty$, we have $1\le \mc{C}(A) \le \mc{C^+}(A)\le 2$. The proof in Section \ref{sec:cramer} and an approximation argument (\cite[Section 4.1]{CMS}) show that, to optimize the value of the constant in Theorem \ref{thm:cramer}, we must find $\mc{C}^+(28)$, where 28 is the constant in the Brun-Titchmarsh-type result given in Corollary \ref{cor:brun}. 
	
\noindent 	One way to construct good functions for some Fourier optimization problems is to consider those of the form $F(x)=P(x)e^{-\pi x^2}$, where $P$ is a polynomial. They were constructed in \cite{CPL} via semidefinite programming. Note, however, that when $A=\infty,$ these functions do not even belong to the family $\mc{E}$, as they are never bandlimited. Similarly, when $A\to \infty$, optimizing $J_A(F)$ requires an increasing concentration of the mass of $\widehat{F}$ in the interval $[-1,1]$. For the same reason, by the uncertainty principle, we might expect that functions of the form $P(x)e^{-\pi x^2}$, when $P$ has bounded degree, become inadequate as $A$ grows, while bandlimited functions of the form \eqref{eq:H(x)}, which give the best known bounds when $A=\infty$ (see \cite{Gor}), become better. This qualitative observation can be formalized in the following way:
	
	\begin{proposition}
		Let $n\ge 1$ be an integer. Let $\mc{F}_n$ be the class of functions of the form $P(x)e^{-\pi x^2}$, where $P\in\R[x]$ is a polynomial of degree at most $n$ (not identically 0). Then, there exists $A_n>1$, such that, for all $A\ge A_n$, we have 
		\begin{equation*}
		\sup_{F\in \mc{F}_n} J_A(F) \leq 0.
		\end{equation*}
	\end{proposition}
\noindent	In particular, for large $A$, polynomials of bounded degree times a gaussian are always far from the (positive) supremum.
	\begin{proof}
		Note that $\mc{F}_n\cup\{0\}$ is a vector space of dimension $n+1$, and it is invariant under the Fourier transform.  Clearly, for any interval $I\subset \R$, the function 
		\[(a_0,\, a_1,\ldots,\, a_{n})\mapsto \int_{I}
		\left|\sum_{j=0}^{n} a_jx^j \right|e^{-\pi x^2} \d x
		\]
		is a continuous function from $\R^{n+1}$ to $\R$, and homogeneous of degree $1$. Therefore, by a compactness argument, there exists a function $F_0\in \mc{F}_n$ that maximizes the quantity 
		\[D_n:= \max_{F\in \mc{F}_n} \frac{\int_{-1}^1 |F(x)|\,\d x}{\int_\R |F(x)|\,\d x} = \max_{F\in \mc{F}_n} \frac{\int_{-1}^1 |\widehat{F}(t)|\,\d t}{\int_\R |\widehat{F}(t)|\,\d t}.
		\]
		Since $F_0$ is not bandlimited, we have $0<D_n<1$. Additionally, note that, for $F\in\mc{F}_n$, we have
		\[|F(0)|\le \int_\R |\widehat{F}(t)|\,\d t \le \frac{1}{1-D_n}\int_{[-1,1]^c}|\widehat{F}(t)|\,\d t.
		\]
		Therefore, for $A>\frac{1}{1-D_n}$, and $F\in\mc{F}_n$, we have $J_A(F)< 0$, and this implies the desired result.
	\end{proof}
	
\vspace{0.1cm}
	
	We conjecture that a similar behavior holds for the problem $\mc{C}^+(A)$, as $A\to\infty$. For instance, functions constructed by David de Laat via semidefinite programming (applying the methods used in \cite{CPL}), with polynomials of degree at most $122$, imply the estimate $\mc{C}^+(28)\geq 1.0865$. Meanwhile, the bandlimited function defined in \eqref{eq:finalInequalityF3} gives $\mc{C}^+(28)\geq 1.0889$. 
	
	In general, for some values\footnote{\,\,\, From \cite[Theorem 1.2]{CMS}, it is known that $\mc{C}^+(1)$=2, and we include our bounds for the sake of comparison. For the other values of $A$, our bounds in Table $1$ slightly improve the general lower bounds obtained in \cite[Theorem 1.2 and 1.3]{CMS}.
} of $A$, Table $1$ compares the lower bounds for $\mc{C}^+(A)$ that are obtained via semidefinite programming, with those obtained using bandlimited functions. The functions constructed via semidefinite programming (following \cite[Section 4]{CPL}) have the form $P(x)e^{-\pi x^2}$, where $P$ is a polynomial of degree at most $82$ or $122$ (that is, functions in $\mc{F}_{82}$ or $\mc{F}_{122}$). On the other hand, the aforementioned bandlimited functions $F$ are constructed as in \eqref{eq:finalInequalityF3} (that is, $F\in \mc{PW}$).\footnote{\,\,\, The notation $\mc{PW}$ comes from the Paley-Wiener space.} Table $2$ gives the necessary parameters to define these functions. They have the form
\begin{align} \label{18_19}
	F(x)=H\bigg(\dfrac{x}{\lambda}\bigg),
\end{align}
	where
\begin{align} \label{18_20}
	H(x)=\cos(2\pi x)\bigg(\frac{a_1}{1-16x^2}+\frac{a_2}{9-16x^2}+\frac{a_3}{25-16x^2}\bigg),
\end{align} 
	with $a_1, a_2, a_3\in \R$. This gives strong evidence for the following conjecture:

\medskip

\begin{conjecture}
    There exists an absolute $\varepsilon>0$, such that the following holds: for $n\ge 1$ an integer, there exists $A_n^+>1$, such that, for $A\ge A_n^+$, we have 
		\begin{equation*}
		\sup_{F\in \mc{F}_n} J_A^+(F) \leq \mc{C}^+(A)-\varepsilon.
		\end{equation*}
\end{conjecture}

\begin{center}
\begin{table} \label{table1}
	\begin{center}
		\begin{tabular}{|c|c|c|c||c|c|c|c|}
			\hline
			$A$ & $\mc{C}^+(A): \mc{F}_{82}$ & $\mc{C}^+(A): \mc{F}_{122}$ & $\mc{C}^+(A): \mc{PW}$ & $A$ & $\mc{C}^+(A): \mc{F}_{82}$ & $\mc{C}^+(A): \mc{F}_{122}$ & $\mc{C}^+(A): \mc{PW}$   \\
			\hline 
			1.0  & 1.9016 & 1.9307 & 1.9602 &
			18.0 & 1.0893 & 1.0944 & 1.0931  \\ 
			\hline
            1.5  & 1.4070 & 1.4089 & 1.3430 & 
			18.5 & 1.0887 & 1.0938 & 1.0928 \\ 
			\hline
            2.0  & 1.2900 & 1.2933 & 1.2417 & 
			19.0 & 1.0881 & 1.0933 & 1.0925  \\ 
			\hline
            2.5  & 1.2346 & 1.2378 & 1.1972 & 
			19.5 & 1.0875 & 1.0928 & 1.0922  \\ 
			\hline
			3.0  & 1.2025 & 1.2049 & 1.1719 & 
			20.0 & 1.0870 & 1.0923 & 1.0919  \\ 
			\hline
			3.5  & 1.1807 & 1.1830 & 1.1555 & 
			20.5 & 1.0865 & 1.0918 & 1.0917  \\ 
			\hline
			4.0  & 1.1653 & 1.1673 & 1.1439 & 
			21.0 & 1.0860 & 1.0914 & 1.0914 \\ 
			\hline
			4.5  & 1.1538 & 1.1555 & 1.1355 & 
			21.5 & 1.0856 & 1.0909 & 1.0912 \\ 	
			\hline
			5.0  & 1.1448 & 1.1467 & 1.1290 & 
			22.0 & 1.0852 & 1.0905 & 1.0909 \\ 
			\hline
			5.5  & 1.1378 & 1.1396 & 1.1239 & 
			22.5 & 1.0848 & 1.0901 & 1.0907 \\ 
			\hline
			6.0  & 1.1320 & 1.1339 & 1.1198 & 
			23.0 & 1.0845 & 1.0897 & 1.0905  \\ 
			\hline
			6.5  & 1.1271 & 1.1294 & 1.1164 & 
			23.5 & 1.0841 & 1.0893 & 1.0903  \\ 
			\hline
			7.0  & 1.1228 & 1.1255 & 1.1136 & 
			24.0 & 1.0838 & 1.0890 & 1.0901  \\ 
			\hline
			7.5  & 1.1191 & 1.1222 & 1.1112 & 
			24.5 & 1.0835 & 1.0886 & 1.0900  \\ 
			\hline
			8.0  & 1.1159 & 1.1192 & 1.1091 & 
			25.0 & 1.0832 & 1.0883 & 1.0898 \\ 
			\hline
			8.5  & 1.1131 & 1.1166 & 1.1073 & 
			25.5 & 1.0830 & 1.0880 & 1.0896  \\ 
			\hline
			9.0  & 1.1107 & 1.1142 & 1.1058 & 
			26.0 & 1.0827 & 1.0876 & 1.0895  \\ 
			\hline
			9.5  & 1.1086 & 1.1121 & 1.1044 & 
			26.5 & 1.0825 & 1.0873 & 1.0893  \\ 
			\hline
			10.0 & 1.1067 & 1.1101 & 1.1031 & 
			27.0 & 1.0823 & 1.0871 & 1.0892  \\ 
			\hline
			10.5 & 1.1049 & 1.1084 & 1.1020 & 
			27.5 & 1.0820 & 1.0868 & 1.0890  \\ 
			\hline
			11.0 & 1.1033 & 1.1068 & 1.1010 & 
			28.0 & 1.0818 & 1.0865 & 1.0889  \\ 
			\hline
			11.5 & 1.1019 & 1.1054 & 1.1001 & 
			28.5 & 1.0816 & 1.0863 & 1.0888  \\ 
			\hline
			12.0 & 1.1005 & 1.1041 & 1.0993 & 
			29.0 & 1.0814 & 1.0860 & 1.0886  \\ 
			\hline
			12.5 & 1.0992 & 1.1030 & 1.0985 & 
			29.5 & 1.0812 & 1.0858 & 1.0885  \\ 
			\hline
			13.0 & 1.0980 & 1.1019 & 1.0978 & 
			30.0 & 1.0810 & 1.0856 & 1.0884 \\ 
			\hline
			13.5 & 1.0969 & 1.1009 & 1.0972 & 
			30.5 & 1.0809 & 1.0854 & 1.0883  \\ 
			\hline
			14.0 & 1.0959 & 1.1000 & 1.0966 & 
			31.0 & 1.0807 & 1.0852 & 1.0882  \\ 
			\hline
			14.5 & 1.0949 & 1.0992 & 1.0960 & 
			31.5 & 1.0805 & 1.0850 & 1.0881 \\ 
			\hline
			15.0 & 1.0940 & 1.0984 & 1.0955 & 
			32.0 & 1.0804 & 1.0848 & 1.0880  \\ 
			\hline
			15.5 & 1.0931 & 1.0976 & 1.0951 & 
			32.5 & 1.0802 & 1.0847 & 1.0879  \\ 
			\hline
			16.0 & 1.0922 & 1.0969 & 1.0946 & 
			33.0 & 1.0800 & 1.0845 & 1.0878  \\ 
			\hline
			16.5 & 1.0915 & 1.0962 & 1.0942 & 
			33.5 & 1.0799 & 1.0844 & 1.0877 \\ 
			\hline
			17.0 & 1.0907 & 1.0956 & 1.0938 & 
			34.0 & 1.0797 & 1.0842 & 1.0876  \\ 
			\hline
			17.5 & 1.0900 & 1.0950 & 1.0935 & 
			34.5 & 1.0796 & 1.0841 & 1.0875 \\ 
			\hline
			\end{tabular}
		\vspace{0.2cm}
		\caption{Table of lower bounds for $\mc{C}^+(A)$ via semidefinite programming and bandlimited functions.}
	\end{center}
\end{table}

\begin{table} \label{table2}
	\begin{center}
		\begin{tabular}{|c|c|c|c||c|c|c|c|}
			\hline
			$A$ & $\mc{C}^+(A)\hspace{0.12cm} \mbox{in} \hspace{0.12cm} \mc{PW}$ & $\{a_1, a_2, a_3\}$ & $\lambda$ & $A$ & $\mc{C}^+(A)\hspace{0.12cm} \mbox{in} \hspace{0.12cm} \mc{PW}$ & $\{a_1, a_2, a_3\}$ & $\lambda$ \\
			\hline 
 1.0 & 1.9602 & \{81, -69, 0\} & 0.100000 & 18.0 & 1.0931 & \{297, 18, 1\} & 0.977220 \\
\hline
1.5 & 1.3430 & \{189, -63, -20\} & 0.660234 &  18.5 & 1.0928 & \{297, 18, 1\} & 0.977843 \\
\hline
 2.0 & 1.2417 & \{243, -57, -20\} & 0.765530 &  19.0 & 1.0925 & \{270, 18, 1\} & 0.978433 \\
\hline
 2.5 & 1.1972 & \{216, -39, -20\} & 0.819517 &  19.5 & 1.0922 & \{270, 18, 1\} & 0.978992 \\
\hline
 3.0 & 1.1719 & \{216, -27, -20\} & 0.852929 &  20.0 & 1.0919 & \{270, 18, 1\} & 0.979523 \\
\hline
 3.5 & 1.1555 & \{216, -18, -20\} & 0.875775 &  20.5 & 1.0917 & \{270, 18, 2\} & 0.980027 \\
\hline
4.0 & 1.1439 & \{243, -15, -20\} & 0.892422 &  21.0 & 1.0914 & \{270, 18, 2\} & 0.980508 \\
\hline
4.5 & 1.1355 & \{270, -9, -20\} & 0.905109 &  21.5 & 1.0912 & \{270, 18, 2\} & 0.980966 \\
\hline
5.0 & 1.1290 & \{297, -6, -20\} & 0.915104 &  22.0 & 1.0909 & \{270, 18, 2\} & 0.981402 \\
\hline
5.5 & 1.1239 & \{324, -3, -20\} & 0.923186 & 22.5 & 1.0907 & \{270, 18, 2\} & 0.981820 \\ 
 \hline
 6.0 & 1.1198 & \{378, 0, -20\} & 0.929858 &  23.0 & 1.0905 & \{270, 18, 2\} & 0.982219 \\
 \hline
 6.5 & 1.1164 & \{405, 3, -20\} & 0.935461 &  23.5 & 1.0903 & \{270, 18, 2\} & 0.982600 \\
 \hline
7.0 & 1.1136 & \{243, 3, -10\} & 0.940232 &  24.0 & 1.0901 & \{270, 18, 3\} & 0.982966 \\
\hline
 7.5 & 1.1112 & \{297, 6, -12\} & 0.944345 &  24.5 & 1.0900 & \{243, 18, 2\} & 0.983317 \\
\hline
 8.0 & 1.1091 & \{270, 6, -9\} & 0.947928 &  25.0 & 1.0898 & \{243, 18, 3\} & 0.983653 \\
 \hline
 8.5 & 1.1073 & \{216, 6, -7\} & 0.951076 &  25.5 & 1.0896 & \{243, 18, 3\} & 0.983976 \\
 \hline
 9.0 & 1.1058 & \{297, 9, -8\} & 0.953865 &  26.0 & 1.0895 & \{243, 18, 3\} & 0.984287 \\
 \hline
 9.5 & 1.1044 & \{270, 9, -7\} & 0.956353 &  26.5 & 1.0893 & \{297, 21, 4\} & 0.984586 \\
 \hline 
 10.0 & 1.1031 & \{243, 9, -5\} & 0.958586 & 27.0 & 1.0892 & \{297, 21, 4\} & 0.984874 \\
 \hline
 10.5 & 1.1020 & \{297, 12, -6\} & 0.960601 & 27.5 & 1.0890 & \{297, 21, 4\} & 0.985151 \\
 \hline
 11.0 & 1.1010 & \{270, 12, -5\} & 0.962429 & 28.0 & 1.0889 & \{68, 5, 1\} & 0.986440 \\
 \hline
 11.5 & 1.1001 & \{270, 12, -4\} & 0.964095 &  28.5 & 1.0888 & \{297, 21, 4\} & 0.985676 \\
 \hline
 12.0 & 1.0993 & \{243, 12, -3\} & 0.965619 &  29.0 & 1.0886 & \{297, 21, 4\} & 0.985924 \\
 \hline
 12.5 & 1.0985 & \{243, 12, -3\} & 0.967019 &  29.5 & 1.0885 & \{297, 21, 4\} & 0.986165 \\
 \hline
 13.0 & 1.0978 & \{297, 15, -3\} & 0.968309 &  30.0 & 1.0884 & \{270, 21, 4\} & 0.986397 \\
 \hline
 13.5 & 1.0972 & \{297, 15, -2\} & 0.969502 &  30.5 & 1.0883 & \{270, 21, 4\} & 0.986622 \\
 \hline
 14.0 & 1.0966 & \{270, 15, -2\} & 0.970609 & 31.0 & 1.0882 & \{270, 21, 4\} & 0.986839 \\
 \hline
 14.5 & 1.0960 & \{270, 15, -1\} & 0.971638 &  31.5 & 1.0881 & \{270, 21, 4\} & 0.987049 \\
 \hline
 15.0 & 1.0955 & \{270, 15, -1\} & 0.972597 & 32.0 & 1.0880 & \{270, 21, 4\} & 0.987253 \\
 \hline
 15.5 & 1.0951 & \{270, 15, -1\} & 0.973494 &  32.5 & 1.0879 & \{270, 21 ,4\} & 0.987450 \\
 \hline
 16.0 & 1.0946 & \{243, 15, 0\} & 0.974334 &   33.0 & 1.0878 & \{270, 21, 4\} & 0.987642 \\
 \hline
 16.5 & 1.0942 & \{243, 15, 0\} & 0.975122 &  33.5 & 1.0877 & \{270, 21, 4\} & 0.987827 \\
 \hline
 17.0 & 1.0938 & \{243, 15, 0\} & 0.975863 &  34.0 & 1.0876 & \{270, 21, 4\} & 0.988007 \\
\hline
 17.5 & 1.0935 & \{297, 18, 0\} & 0.976561 &  34.5 & 1.0875 & \{270, 21, 4\} & 0.988182 \\ \hline
			\end{tabular}
		\vspace{0.2cm}
		\caption{Table of lower bounds for $\mc{C}^+(A)$ via bandlimited functions, with the corresponding parameters as defined in \eqref{18_19} and \eqref{18_20}.}
	\end{center}
\end{table}
\end{center}

\noindent 	However, proving it seems more subtle, and is related to the concentration of positive mass of a function, instead of total mass. Sign uncertainty principles of this type were first considered by Bourgain, Clozel, and Kahane \cite{BCK}, and have recently attracted considerable attention (see, for instance, \cite{CQ,CG,GOR1,GOR2}). In particular, a similar conjecture for the functions $P(x)e^{-\pi x^2}$ of bounded degree was stated in \cite[Conjecture 3.2]{CG}.

\newpage
	\section*{Appendix: Some useful estimates}

	\begin{lemma} \label{23_54} Let $x,y\geq 1$ be two parameters. Consider the radial function $G:\R^2\to \R$ defined by
		\begin{equation*} 
		G(r)=\left\{
		\begin{array}{ll}
		\min\bigg\{r^2,1,\dfrac{x+y-r^2}{y}\bigg\},\hspace{0.3cm}\text{if}\hspace{0.2cm} 0\leq r\leq (x+y)^{1/2}; \\
		0,\hspace{3.75cm}\text{if}\hspace{0.2cm}  r> (x+y)^{1/2}.
		\end{array}
		\right.
		\end{equation*}
		Then $G\in L^1(\R^2)$, and its Fourier transform $\widehat{G}$ satisfies the following properties:\begin{enumerate}
			\item For $\xi\in \R^2$ and $\xi\neq 0$,
			\begin{equation} \label{eq:H11}
			\big|\widehat{G}(\xi)\big| \ll \dfrac{(x+y)^{1/4}}{|\xi|^{3/2}}.
			\end{equation}
			\item For $\xi\in \R^2$ and $|\xi|\geq 1$,
			\begin{equation} \label{eq:H22}
			\big|\widehat{G}(\xi)\big| \ll \dfrac{1}{|\xi|^{5/2}}\bigg(1+\dfrac{x^{3/4}}{y}\bigg).
			\end{equation}
			\item For $\xi=0$,
			\begin{equation} \label{21_10}
			\widehat{G}(0)=\bigg(x+\dfrac{y}{2}\bigg)\pi -\dfrac{\pi}{2},
			\end{equation} 
		\end{enumerate}
	\end{lemma}
	\begin{proof} It is clear that $G\in L^1(\R^2)$. Since $G$ is a radial function, it follows that (see \cite[p. 429]{Grafakos}),
		\begin{align} \label{14_51} 
		\widehat{G}(\xi) = 2\pi\int_{0}^\infty r\,G(r)\,J_0(2\pi r|\xi|)\,\d r,
		\end{align} 
		for $\xi\in \R^2$, where $J_0$ is the Bessel function of order $0$. Since $J_0(0)=1$, a simple computation shows \eqref{21_10}. Let us start proving the estimate \eqref{eq:H11}. For $\xi\neq 0$, we split the integral in \eqref{14_51} into the ranges $0\leq r <1/|\xi|$ and $1/|\xi|\leq r <\infty$. Using the estimates $|J_0(t)|\ll 1$ and $|G(r)|\leq 1$, it follows that 
		\begin{align} \label{19_31}
		\bigg|\int_{0}^{1/|\xi|} r\,G(r)\,J_0(2\pi r|\xi|)\,\d r \bigg| \ll \dfrac{1}{|\xi|^2}. 
		\end{align}
		To estimate the second integral, by \cite[8.451-1]{GR} we recall that 
		$$
		J_0(t)= \bigg(\dfrac{2}{\pi t}\bigg)^{1/2}\bigg\{\cos(t-\tfrac{\pi}{4})+\frac{1}{8t}\sin(t-\tfrac{\pi}{4})+O\bigg(\dfrac{1}{t^2}\bigg)\bigg\}
		$$
		for $|t|\gg 1$. Then, using integration by parts and the facts that $|G(r)|\leq 1$, and $\int_{0}^{\infty }G(r)/r^{3/2}\,\d r<\infty$, we obtain that
		\begin{align} \label{17_29}
		\begin{split}
		\int_{1/|\xi|}^\infty r\,G(r)\,&J_0(2\pi r|\xi|) \,\d r \\
		& =-\dfrac{1}{2\pi^2|\xi|^{3/2}}\int_{1/|\xi|}^\infty (r^{1/2}G(r))'\,\sin(2\pi r|\xi|-\tfrac{\pi}{4})\,\d r  \\ 
		& \,\,\,\,\,\,+ \dfrac{1}{32\pi^3|\xi|^{5/2}}\int_{1/|\xi|}^\infty \bigg(\dfrac{G(r)}{r^{1/2}}\bigg)'\,\cos(2\pi r|\xi|-\tfrac{\pi}{4})\,\d r+ O\bigg(\min\bigg\{\dfrac{1}{|\xi|^{5/2}},\dfrac{1}{|\xi|^{2}}\bigg\}\bigg).
		\end{split}
		\end{align}
		Therefore,
		\begin{align*}  
		\bigg|\int_{1/|\xi|}^\infty r\,G(r)\,J_0(2\pi r|\xi|)\,\d r \bigg| & \ll\dfrac{1}{|\xi|^{3/2}}\int_{1/|\xi|}^\infty \big|(r^{1/2}G(r))'\big|\,\d r + \dfrac{1}{|\xi|^{5/2}}\int_{1/|\xi|}^\infty \bigg|\bigg(\dfrac{G(r)}{r^{1/2}}\bigg)'\bigg|\,\d r+ \min\bigg\{\dfrac{1}{|\xi|^{5/2}},\dfrac{1}{|\xi|^{2}}\bigg\} \\
		& \ll \dfrac{1}{|\xi|^{3/2}}\int_{0}^\infty \bigg|\dfrac{G(r)}{r^{1/2}}\bigg|\,\d r + \dfrac{1}{|\xi|^{3/2}}\int_{0}^\infty \big|r^{1/2}G'(r)\big|\,\d r + \dfrac{1}{|\xi|^{2}}.
		\end{align*}
		Spliting the above integrals according to the definition of $G$, and using the mean value theorem, it follows that, for $1/|\xi|\leq \sqrt{x+y}$, we have
		\begin{equation} \label{19_32}
		\bigg|\int_{1/|\xi|}^\infty r\,G(r)\,J_0(2\pi r|\xi|)\,\d r \bigg| \ll \dfrac{(x+y)^{1/4}}{|\xi|^{3/2}}.
		\end{equation}
		Combining \eqref{19_31} and \eqref{19_32} we obtain \eqref{eq:H11} in the case $1/|\xi|\leq (x+y)^{1/2}$. When $1/|\xi|> (x+y)^{1/2}$, we bound as in \eqref{19_31} to obtain 
		\begin{align*} 
		\big|\widehat{G}(\xi)\big|=\bigg|2\pi\int_{0}^{(x+y)^{1/2}} r\,G(r)\,J_0(2\pi r|\xi|)\,\d r \bigg| \ll x+y\ll \dfrac{(x+y)^{1/4}}{|\xi|^{3/2}}.
		\end{align*}
		This conclude the proof of the estimate \eqref{eq:H11}. Now, let us prove \eqref{eq:H22}.  Suppose that $|\xi|\geq 1$. We split the integral in \eqref{14_51} as in the previous case, and we bound the first integral as follows:
		\begin{align} \label{19_311}
		\bigg|\int_{0}^{1/|\xi|} r\,G(r)\,J_0(2\pi r|\xi|)\,\d r \bigg| \ll  \int_{0}^{1/|\xi|} r^3\,\d r \ll \dfrac{1}{|\xi|^4}. 
		\end{align}
		On the other hand, in \eqref{17_29} we split the last integrals (depending on the value of $1/|\xi|\leq 1$) and use integration by parts (one more time). In this way, we obtain that 
		\begin{align} \label{19_312}
		\bigg|\int_{1/|\xi|}^\infty r\,G(r)\,J_0(2\pi r|\xi|) \,\d r\bigg|
		\ll \dfrac{1}{|\xi|^{5/2}}\bigg(1+\dfrac{x^{3/4}}{y}\bigg).
		\end{align}
		Then, combining \eqref{19_311} and \eqref{19_312} we obtain \eqref{eq:H22}.
	\end{proof}

	\begin{lemma} \label{lemma10}
		Let $K$ be an imaginary quadratic field, and let $\qq$ be an integral ideal. Then,
		\begin{equation*}
		\sum_{\mathfrak{p}\mid \mathfrak{q},\, k\ge 1}\frac{\log \rm N\mathfrak{p}}{(\rm N\mathfrak{p})^{k/2}}\ll 
		\sqrt{\log \rm N\mathfrak{q}}.
		\end{equation*}
	\end{lemma}
	\begin{proof}
		Using the factorization law of primes in imaginary quadratic fields \cite[p. 57]{IK}, one can see that, for each $k\geq 1$,
		\begin{align*} 
		\sum_{\mathfrak{p}\mid \mathfrak{q}}\frac{\log \rm N\mathfrak{p}}{(\rm N\mathfrak{p})^{k/2}}& =\sum_{p}\left(\sum_{\substack{\mathfrak{p}\mid \mathfrak{q}\\\rm N\pp=p}}\frac{\log \rm N\mathfrak{p}}{(\rm N\mathfrak{p})^{k/2}}\right)+\sum_{p}\left(\sum_{\substack{\mathfrak{p}\mid \mathfrak{q}\\\rm N\pp=p^2}}\frac{\log \rm N\mathfrak{p}}{(\rm N\mathfrak{p})^{k/2}}\right)\ll \sum_{p|\rm N\qq}\frac{\log p}{p^{k/2}}.
		\end{align*} 
		It is clear that the sum over $k\geq 3$ in the above expression contributes $O(1)$, and the sum when $k=2$ is bounded by the sum when $k=1$. Let us analyze the latter case. Assume that $\rm N\qq\geq 3$. We denote by $\omega(n)$ the number of distinct positive integer prime factors of $n$, and by $p_n$ the $n$-th prime number. Since $p_n\leq Cn\log n$ for some $C>0$, and the function $y\mapsto y^{-k/2}\log y$ is eventually decreasing, it follows that
		$$
		\sum_{p|\rm N\qq}\frac{\log p}{\sqrt{p}} \ll \sum_{p\leq p_{\omega(\rm N\qq)}}\frac{\log p}{\sqrt{p}}\ll\sum_{p\leq C\omega(\rm N\qq)\log(\omega(\rm N\qq))}\frac{\log p}{\sqrt{p}}\ll \sqrt{{\omega(\rm N\qq)}\log(\omega(\rm N\qq))},
		$$
		where we used integration by parts in the last step. We conclude our desired result using the classical estimate for $\omega(n)$ (see \cite[Theorem 2.10]{MV2}):
		$$
		w(n)\ll \dfrac{\log n}{\log\log n}.
		$$
	\end{proof}

	\section*{Acknowledgements}
	We would like to thank David de Laat for carrying out computations to obtain the lower bound for $\mc{C}^+(28)$ and Table $1$, via semidefinite programming. We are grateful to Emanuel Carneiro, Micah B. Milinovich, Kristian Seip, Jesse Thorner, Asif Zaman, and the anonymous referee for their helpful comments. AC was supported by Grant 275113 of the Research Council of Norway. O.Q-H. acknowledges support from CNPq - Brazil and from the STEP program of ICTP - Italy.

\end{document}